\documentclass[12pt,reqno]{amsart}

\usepackage{amsmath, amssymb, amsthm}

\usepackage[dvipdfmx]{graphicx}

\usepackage{xcolor}
\usepackage{hyperref}
\usepackage[numbers]{natbib}

\theoremstyle{plain}
\newtheorem{theorem}{Theorem}[section]
\newtheorem{lemma}[theorem]{Lemma}
\newtheorem{proposition}[theorem]{Proposition}
\newtheorem{corollary}[theorem]{Corollary}

\theoremstyle{remark}
\newtheorem{remark}{Remark}

\numberwithin{equation}{section}

\newcommand{\C}{\mathbb{C}}
\newcommand{\D}{\mathbb{D}}
\newcommand{\N}{\mathbb{N}}
\newcommand{\R}{\mathbb{R}}
\newcommand{\T}{\mathbb{T}}
\newcommand{\Z}{\mathbb{Z}}

\date{\today}

\title[\null]{Local approximations of inverse block Toeplitz matrices and Baxter-type theorems for long-memory processes}

\author[\null]{AKIHIKO INOUE and JUNHO YANG}

\address{A.\ Inoue\\
Department of Mathematics\\
Hiroshima University\\
Higashi-Hiroshima 739-8526\\
Japan.}
\email{inoue100@hiroshima-u.ac.jp}

\address{J.\ Yang\\
Institute of Statistical Science \\
Academia Sinica \\
Taipei 115 \\
Taiwan.}
\email{junhoyang@stat.sinica.edu.tw}

\begin{document}

\maketitle

\begin{abstract}
We derive sharp approximation error bounds for inverse block Toeplitz matrices associated with multivariate long-memory stationary processes.
The error bounds are evaluated for both column and row sums. These results are used to prove the strong convergence of the solutions of general block Toeplitz systems. A crucial part of the proof is to bound sums consisting of the Fourier coefficients of the phase function attached to the singular symbol of the Toeplitz matrices.

\vspace{0.5em}


\end{abstract}


\section{Introduction}\label{sec:1}

Let $q\in\N$ and let $\T:=\{z\in\C :\vert z\vert=1\}$ be the unit circle in $\C$. 
Consider a $\C^{q}$-valued second-order stationary time series $\{X_k:k\in\Z\}$ with spectral density $w:\T \to \C^{q\times q}$ and autocovariance function $\gamma: \Z \to \C^{q \times q}$. Thus, we have
\begin{equation}
\gamma(k)= \mathrm{Cov}(X_k, X_0) = 
\int_{-\pi}^{\pi}e^{-ik\theta}w(e^{i\theta})\frac{d\theta}{2\pi}, 
\quad k\in\Z.
\label{eq:gamma123}
\end{equation}
Now, we consider the semi-infinte block Toeplitz matrix
\begin{equation}
T_{\infty}(w) :=
\left(
\begin{matrix}
\gamma(0) & \gamma(-1) & \gamma(-2) & \cdots \cr
\gamma(1) & \gamma(0) & \gamma(-1) & \cdots \cr
\gamma(2) & \gamma(1) & \gamma(0) & \cdots \cr
\vdots & \vdots & \vdots & \ddots 
\end{matrix}
\right)
\label{eq:Tinfty135}
\end{equation}
with symbol $w$, which is the covariance matrix of $(X_1^\top, X_2^\top, \dots)^\top$.
For $n\in\N$, we also consider the truncated block Toeplitz matrix
\[
T_{n}(w) := (\gamma(s-t))_{1\leq s,t \leq n} \in \C^{qn \times qn},
\]
which is the covariance matrix of $(X_1^\top, \dots, X_n^\top)^\top$. 


Covariance matrices and their inverses (also known as \textit{precision matrices}) are widely used in statistics and applied probability. Applications include quasi-Gaussian likelihood and graphical models, among others. 
For a survey, see \cite{p:gray-06}, and for a textbook treatment,  \cite{b:gre-58}. 
While the covariance matrix $T_n(w)$ has a simple (block) Toeplitz structure, its inverse $T_n(w)^{-1}$ poses theoretical and computational challenges.
In contrast, the entries of the inverse $T_\infty(w)^{-1}$ of $T_\infty(w)$ in (\ref{eq:Tinfty135}) allow a simple closed-form expression in terms of the (backward) infinite-order autoregressive coefficients of $w$ (see (\ref{eq:Tinftyinv973}) and Section \ref{sec:4} below). 
Thus, $T_\infty(w)^{-1}$ provides a good theoretical and practical basis for studying the properties of $T_n(w)^{-1}$. In this context, it is important to investigate how well $T_{\infty}(w)^{-1}$ approximates $T_n(w)^{-1}$.

Let $\overline{\D}:=\{z\in\C :\vert z\vert\le 1\}$ be the closed unit disk in $\C$. 
In this paper, we restrict to the symbol $w$ satisfying the 
following condition (F$_d$) for $d \in (0,1/2)$:
\[
\mbox{$w(e^{i\theta})=|1-e^{i\theta}|^{-2d}g(e^{i\theta})g(e^{i\theta})^*$, 
where $g:\T\to \C^{q\times q}$ satisfies (C).}
\tag{F$_d$}
\]
Here, the condition (C) for $g:\T \to \C^{q\times q}$ is defined as follows:
\[
\begin{aligned}
&\mbox{the entries of $g(z)$ are rational functions in $z$ that have 
no poles in $\overline{\D}$, }\\
&\mbox{and $\det g(z)$ has no zeros in $\overline{\D}$.}
\end{aligned}
\tag{C}
\]
Thus, $w(e^{i\theta})$ has a unique singularity at $\theta=0$. 
In time series literature, the corresponding second-order stationary process $\{X_k\}$ with 
spectral density $w$ of the form (F$_d$) for $d\in (-1/2,1/2)\setminus \{0\}$ is called the $q$-variate \textit{autoregressive fractionally integrated moving-average} (ARFIMA) process or the \textit{vector ARFIMA} process. 
In particular, (F$_d$) for $d\in(0,1/2)$ implies that $\{X_k\}$ has \textit{long memory} in the sense that 
$\sum_{k=-\infty}^{\infty} \Vert \gamma(k) \Vert= \infty$ 
holds (see, e.g., \cite{I02} and \cite{IKP2}, and the references therein). 
Here, for $a \in \C^{q\times q}$, $\|a\|$ denotes the spectral norm of $a$. 

To investigate the error bound between $T_{\infty}(w)^{-1}$ and $T_n(w)^{-1}$ corresponding to the vector ARFIMA process, we focus on the column sums
\begin{equation}
C_{t,n} := \sum_{s=1}^{n} \Delta_n^{s,t}, 
\quad n \in \N,\ \ t \in \{1,\dots, n\},
\label{eq:RCn-a}
\end{equation}
where $\Delta_n^{s,t}$ denotes the following norm of the block-wise difference:
\begin{equation}
\Delta_n^{s,t} := \Vert (T_n(w)^{-1})^{s,t} - (T_{\infty}(w)^{-1})^{s,t}\Vert, 
\quad 
n \in \N,\ \ s, t \in \{1,\dots, n\}.
\label{eq:Delta}
\end{equation}
Here, for $n\in\N$, $A \in \C^{qn \times qn}$ and $s, t\in\{1,\dots,n\}$, 
we write $A^{s,t}\in \C^{q\times q}$ for the $(s,t)$ block of $A$; thus 
$A = (A^{s,t})_{1\le s, t\le n}$. Even if $A$ is a semi-infinite matrix 
like $T_{\infty}(w)^{-1}$, $A^{s,t}\in \C^{q\times q}$ similarly denotes the $(s,t)$ block of $A$.


Here is our main theorem that addresses the upper bounds of $C_{t,n}$.

\begin{theorem}\label{thm:1}
We assume (F${}_d$) for $d \in (0,1/2)$. Let $\delta \in (0,1)$. 
Then, there exists $M \in (0,\infty)$, which depends only on $w$ and $\delta$, such that
\begin{equation} 
C_{t,n} \le
\begin{cases}
M n^{-d} t ^{-d}, & n \in \N, ~~ t \in \{1, \dots, [\delta n]\}, \\
M n^{(1/2)-d} (n+1-t)^{-(1/2)-d}, & n \in \N, ~~ t \in \{ [\delta n]+1, \dots, n \}. 
\label{eq:col-summ1}
\end{cases}
\end{equation}
Moreover, it holds that
\begin{equation}
\sup_{n \in \N}\max_{t \in \{1, \dots ,n\}}C_{t,n} <\infty. 
\label{eq:col-summ2}
\end{equation}
\end{theorem}

Theorem \ref{thm:1} is new even for $q=1$, i.e., the univariate ARFIMA model with $d\in (0,1/2)$. 
Due to Lemma \ref{lemma:CRequal333} below, Theorem \ref{thm:1} also holds for the row sums $R_{t,n}$ as in (\ref{eq:RCn-b}) replacing $C_{t,n}$.

We show some numerical results to give a better understanding of Theorem \ref{thm:1}. 
Let $w$ be the spectral density of a univariate ARFIMA$(0,d,0)$ process given by $w(e^{i\theta}) = |1-e^{i\theta}|^{-2d}$, $d \in (0,1/2)$.
Using expressions for the infinite-order autoregressive coefficients and the autocovariances as described in \cite{b:ber-13}, pages 47--50, we can derive explicit expressions for $T_n(w)$ and $T_{\infty}(w)^{-1}$ (see (\ref{eq:Tinftyinv973})) for the defined spectral density. The left panel of Figure \ref{fig:1} calculates, for each $d \in \{0.1, 0.2, 0.3, 0.4\}$,
the column-wise sums $C_{t,n}$ for $n=1,000$ and $t \in \{1,\dots,n\}$. The right panel computes $C_{1,n}$ (labled in circle) and $C_{n/2,n}$ (labeled in triangle) for each $n \in \{10,20, \dots, 1000\}$ and for $d=0.4$.
\begin{figure}[h]
\includegraphics[width=0.4\textwidth]{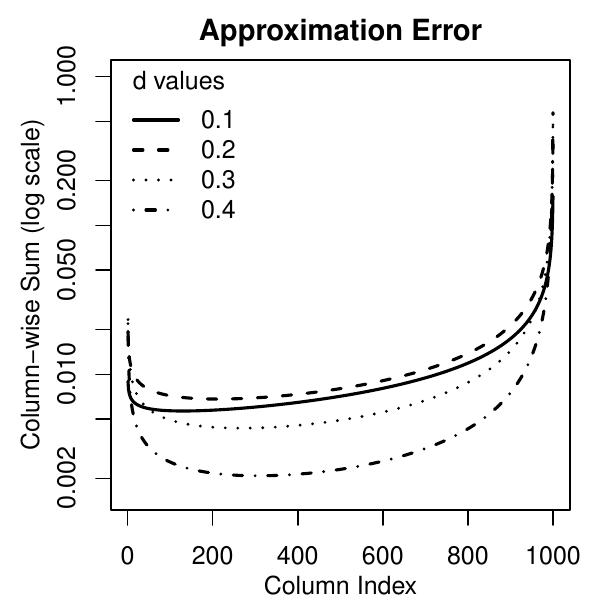}
\includegraphics[width=0.4\textwidth]{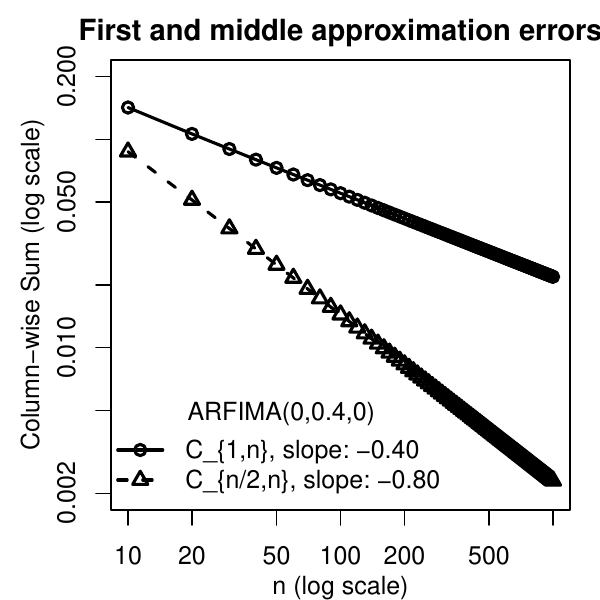}
\centering
\caption{Left: Plot of $C_{t,n}$ for the univariate ARFIMA$(0,d,0)$ spectral density function $w$. We set $n=1,000$ and $d \in \{0.1, 0.2, 0.3, 0.4\}$. Right: Log-log plot of $C_{1,n}$ (circle) and $C_{n/2,n}$ (triangle) for the univariate ARFIMA$(0,0.4,0)$ for each $n \in \{10,20,\dots,1000\}$.}
\label{fig:1}
\end{figure}
Note from (\ref{eq:col-summ1}) that for $t \ll n$, $C_{t,n}$ is uniformly bounded by $M n^{-d} t^{-d}$. This is consistent with the decreasing trend of $C_{t,n}$ as $t$($ \ll n$) increases, with the slope becoming steeper as $d$ increases. 
In the left panel of Figure \ref{fig:1}, when $t$ is close to $n$, $C_{t,n}$ increases, but remains bounded above, even for large $n$. 
This supports the uniform boundedness result in (\ref{eq:col-summ2}). Regarding the ``sharpness'' of the bounds derived in our results, the log-log plot in the right panel of Figure \ref{fig:1} shows that $\lim_{n\rightarrow \infty}n^{d} C_{1,n} = C_1 \in (0,\infty)$ and $\lim_{n\rightarrow \infty}n^{2d} C_{n/2,n} = C_2 \in (0,\infty)$. Thus, at least for the univariate ARFIMA$(0,d,0)$ model under consideration, 
the bounds in (\ref{eq:col-summ1}) appear to be tight.

We turn to the theoretical contributions of our main theorem. A noteworthy aspect of Theorem \ref{thm:1} is that it provides approximation error bounds of $T_n(w)^{-1}$ from a ``local'' perspective, which, as far as we know, is the first attempt. Prior studies on the approximation of inverse Toeplitz matrices include, for example, \cite{p:gal-74, p:sha-75} (for short-memory processes; see (\ref{eq:S}) below) and \cite{p:fox-87, p:dah-89, p:ing-16} (for long-memory processes). These studies derived error bounds in terms of the spectral or Frobenius norm, resulting in a ``global'' evaluation of the approximation.
 
The motivation for our local approximation approach stems from our attempt to prove the Baxter-type theorems for long-memory processes which is our main application to be elaborated below. Another application is alternative approximations of inverse Toeplitz matrices. Initially, we tried to obtain a bound on $\sum_{s,t=1}^{n} \Delta_{n}^{s,t}$, but it seemed that such global approximation schemes only work for the short-memory framework, but not for the long-memory framework. We finally realized that the local bounds in (\ref{eq:col-summ1}) and (\ref{eq:col-summ2}) are crucial for evaluating the approximation errors of finite section Wiener--Hopf solutions for long-memory processes, and thus for proving the desired Baxter-type theorems.

We now elaborate Baxter-type theorems. Let
$\C^{\infty\times q}:=\{ \mathbf{y} = (y_1^\top, y_2^\top, \dots)^\top: y_k\in \C^{q\times q}, k \in \N\}$.
For an output sequence 
$\mathbf{y} = (y_1^\top, y_2^\top,\dots)^\top \in \C^{\infty\times q}$, we write $\mathbf{y}_n$ for the truncation of $\mathbf{y}$ defined by
\begin{equation}
\mathbf{y}_n := (y_1^\top, \dots, y_n^\top)^\top \in \C^{nq\times q}, \quad n \in \N.
\label{eq:yn234}
\end{equation}
Let 
$\mathbf{z} = (z_1^{\top},z_2^{\top},\dots)^{\top} \in \C^{\infty\times q}$ 
be the solution to the semi-infinite block Toeplitz system
\begin{equation}
T_{\infty}(w)\mathbf{z} = \mathbf{y}.
\label{eq:tsysinf123}
\end{equation}
The system of equations (\ref{eq:tsysinf123}) is called the \textit{Wiener--Hopf equation} (\cite{b:wie-49}).
Let 
$\mathbf{z}_n = (z_{1,n}^\top, \dots, z_{n,n}^\top)^\top \in \C^{nq\times q}$ be the solution to 
the finite section Wiener--Hopf equation
\begin{equation}
T_n(w) \mathbf{z}_n = \mathbf{y}_n.
\label{eq:tsysn123}
\end{equation}
In this situation, we are interested in whether
\begin{equation}
\lim_{n\to\infty} \sum_{k=1}^{n} \|z_{k} - z_{k,n}\| = 0
\label{eq:Baxter}
\end{equation}
holds and, if (\ref{eq:Baxter}) holds, the rate of convergence. 
Results of the form (\ref{eq:Baxter}) and the rates of convergence are of great importance in the prediction theory of time series. For example, 
they are closely related to the mean-squared prediction errors (see \cite{p:pou-89} and more recently \cite{p:yang-24}). To put our results in historical perspective, we call results of the form (\ref{eq:Baxter}) {\it Baxter-type theorems}. 
See Baxter's inequality (\ref{eq:Baxter597}) below, which gives a prototype result. 

A Baxter-type theorem for non-singular $w$ has been proved in \cite{I23}, where it is assumed that 
$\mathbf{y}$ is such that $\sum_{k=1}^{\infty} \|y_k\| <\infty$ and that 
$w$, instead of ($\mathrm{F}_d$), satisfies the following condition:
\begin{equation}
\begin{aligned}
&\mbox{$\int_{-\pi}^{\pi} \Vert w(e^{i\theta})\Vert d\theta<\infty$, 
$w(e^{i\theta})$ is a positive Hermitian matrix for $\theta\in (-\pi,\pi]$,}\\
&\mbox{$\sum_{k=-\infty}^{\infty} \Vert \gamma(k)\Vert <\infty$, 
and $\min_{z\in\T}\det w(z)>0$.}
\end{aligned}
\label{eq:S}
\end{equation}
The condition (\ref{eq:S}) implies that the corresponding $q$-variate second-order stationary 
process has {\it short memory}. 
However, the techniques used to prove the Baxter-type thoerem for short memory processes cannot be readily applied to long-memory processes. This is because \cite{I23} takes full advantage of the fact that the sequence of infinite-order moving-average coefficients of $w$ is absolutely summable, a condition not satisfied by $w$ satisfying (F$_d$) with $d \in (0,1/2)$. 

It turns out that to prove Baxter-type theorems for long-memory processes we need a more subtle approach to 
bound the difference $\|z_{k}-z_{k,n}\|$, where the local approximation bounds derived in Theorem \ref{thm:1} play a key role. 
Interestingly, our local approximation approach allows us to derive the convergence rate of 
(\ref{eq:Baxter}). Note that the method in \cite{I23}, which uses the dominated convergence theorem, 
does not allow us to obtain the convergence rate even for short-memory processes.



This paper is organized as follows. 
In Section \ref{sec:3}, we present six theorems (Theorems \ref{thm:Bound1}--\ref{thm:Bound6}) 
on bounds of $R_{s,n}$ and $C_{t,n}$, and using them, we prove Theorem \ref{thm:1}. 
Section \ref{sec:appl} discusses two applications of Theorem \ref{thm:1}: two Baxter-type theorems 
for vector ARFIMA processes with $d\in (0,1/2)$ 
(Section \ref{sec:baxter}) and alternative approximations of inverse Toeplitz matrices (Section \ref{sec:alt}).
Sections \ref{sec:4}--\ref{sec:10} cover the technical details. 
In Section \ref{sec:4}, we derive a key bound for $\Delta_n^{s,t}$. 
In Section \ref{sec:5}, we prove some results on the bounds for sums consisting of the Fourier coefficients of the phase function associated with $w$. 
In Section \ref{sec:6}, we prove the result on the bound for $R_{s,n}$ (Theorem \ref{thm:Bound1}). 
In Section \ref{sec:7}, we prove a result on the bound for $C_{t,n}$ (Theorem \ref{thm:Bound3}). 
In Section \ref{sec:8}, we prove another result on the bound for $C_{t,n}$ (Theorem \ref{thm:Bound5}).
In in Section \ref{sec:9}, we prove the two Baxter-type theorems given in Section \ref{sec:3}. 
Finally, in Section \ref{sec:10}, we prove Theorem \ref{thm:Omega}.



\section{Bounds for sums of $\Delta_n^{s,t}$ and proof of Theorem \ref{thm:1}}\label{sec:3}

In this section, we present some results on the upper bounds of sums of $\Delta_n^{s,t}$ in (\ref{eq:Delta}). 
To do so, we also need to consider the row sums
\begin{equation} 
R_{s,n} := \sum_{t=1}^{n} \Delta_n^{s,t}, 
\quad 
n \in \N,\ \ s \in \{1,\dots, n\}.
\label{eq:RCn-b}
\end{equation}
By the next lemma, we have $R_{t,n} = C_{t,n}$;
however, depending on the situation, we either equate or differentiate between $R_{t,n}$ and $C_{t,n}$. 

\begin{lemma}\label{lemma:CRequal333}
$R_{t,n} = C_{t,n}$ for $n\in \N$ and $t\in \{1, \dots,n\}$.
\end{lemma}

\begin{proof}
Since $T_n(w)$, hence $T_n(w)^{-1}$, is a Hermitian matrix, we have 
$(T_n(w)^{-1})^{s,t} = ((T_n(w)^{-1})^{t,s})^*$ for $s, t \in \{1,\dots,n\}$. 
Similarly, we have 
\begin{equation}
(T_{\infty}(w)^{-1})^{s,t} = ((T_{\infty}(w)^{-1})^{t,s})^*, \quad 
s, t \in \{1,\dots,n\}.
\label{eq:Tinvsa333}
\end{equation}
Therefore, the assertion holds.
\end{proof}

The next two theorems give bounds for the row sums $R_{s,n}$.

\begin{theorem} \label{thm:Bound1}
We assume (F$_{d}$) for $d\in (0,1/2)$.  Then, there exist $M_1 \in (0, \infty)$, 
which depends only on $w$, such that
\begin{equation}
R_{s,n} \leq
M_1  n^{(1/2)-d} \left\{ (n+1-s)^{-(1/2)-d} + s^{-(1/2)-d} \right\}, \quad 
n \in \N, \ \ s \in \{1, \dots,n\}.
\label{eq:Tnwtt}
\end{equation}
\end{theorem}

The proof of Theorem \ref{thm:Bound1} will be given in Section \ref{sec:6}.

For $x \in \R$, let $[x]$ be the floor function or the greatest integer 
less than or equal to $x$.

\begin{theorem}\label{thm:Bound2}
We assume (F${}_d$) for $d\in (0,1/2)$. Let $\delta \in (0,1)$. 
Then there exists $M_2 \in (0,\infty)$, which depends only on $w$ and $\delta$, 
such that
\begin{equation}
R_{t,n}
\le M_2 n^{(1/2)-d} (n+1-t)^{-(1/2)-d}, \quad
n\in\N,\ \ t \in \{ [\delta n]+1,\dots, n \}.
\label{eq:col-sum2}
\end{equation}
\end{theorem}

\begin{proof}
Let $\delta \in (0,1)$ and 
$r := (2/\delta) -1 > 0$. Then we easily find that
$t^{-1} \leq r(n+1-t)^{-1}$ for $n \in \N$ and $t\in \{ [\delta n]+1, \dots, n\}$. 
Therefore, Theorem \ref{thm:Bound1} yields, 
for $n\in \N$ and $ t\in \{ [\delta n]+1, \dots, n\}$, 
\[
R_{t,n} 
\leq M_1 \{r^{(1/2)+d}+1\} n^{(1/2)-d} (n+1-t)^{-(1/2)-d}. 
\]
Therefore, 
(\ref{eq:col-sum2}) holds with $M_2 = M_1 \{r^{(1/2)+d}+1\}$.
\end{proof}

The next three theorems give bounds for the column sums $C_{t,n}$.

\begin{theorem}\label{thm:Bound3}
We assume (F$_{d}$) for $d\in (0,1/2)$.  Then there exist $M_3 \in (0, \infty)$, 
which depends only on $w$, such that
\begin{equation}
C_{t,n}
\leq M_3 n^{1-d} t^{-d} (n+1-t)^{-1}, \quad n \in \N, \ \ t \in\{1, \dots,n\}.
\label{eq:Tnwss}
\end{equation}
\end{theorem}

The proof of Theorem \ref{thm:Bound3} will be given in Section  \ref{sec:7}.

\begin{theorem}\label{thm:Bound4}
We assume (F${}_d$) for $d\in (0,1/2)$. Let $\delta \in (0,1)$. 
Then there exists $M_4 \in (0,\infty)$, which depends only on $w$ and $\delta$, 
such that
\begin{equation}
C_{t,n}
\le M_4 n^{-d} t ^{-d},\quad 
n\in\N,\ \ t \in \{1,\dots, [\delta n]\}.
\label{eq:col-sum1}
\end{equation}
\end{theorem}

\begin{proof}
Let $\delta \in (0,1)$. 
Then, by Theorem \ref{thm:Bound3}, for $n\in \N$ and $ t\in \{1, \dots, [\delta n]\}$,
\begin{equation*}
C_{t,n} \leq M_3 n^{1-d} t^{-d} (n+1-t)^{-1} \leq M_3 n^{1-d} t^{-d} (n-n\delta)^{-1}
= M_3(1-\delta)^{-1} n^{-d} t^{-d}.
\end{equation*}
Thus (\ref{eq:col-sum1}) holds with $M_4 = M_3/(1 - \delta)$.
\end{proof}

\begin{theorem}\label{thm:Bound5}
We assume (F${}_d$) for $d\in (0,1/2)$. 
Then there exists $M_5 \in (0,\infty)$, which depends only on $w$, 
such that
\begin{equation}
C_{t,n} \le M_5, \quad
n\in\N,\ \ t \in \{ 1,\dots, n \}.
\label{eq:col-sum3}
\end{equation}
\end{theorem}

The proof of Theorem \ref{thm:Bound5} will be given in Section \ref{sec:8}.

\vspace{0.5em}

By using Lemma \ref{lemma:CRequal333} and Theorems \ref{thm:Bound2}, \ref{thm:Bound4} and \ref{thm:Bound5}, 
we can prove Theorem \ref{thm:1}.

\begin{proof}[Proof of Theorem \ref{thm:1}]
The assertion (\ref{eq:col-summ2}) is the same as Theorem \ref{thm:Bound5}. 
By (\ref{eq:col-sum1}), (\ref{eq:col-sum2}) and Lemma \ref{lemma:CRequal333}, 
(\ref{eq:col-summ1}) holds with $M = \max(M_2, M_4) \in (0,\infty)$.
\end{proof}

We end this section with a result on the bound for the partial sum in $C_{t,n}$. 
This result will be used in Section \ref{sec:9} to prove Theorem \ref{thm:Omega} below.

\begin{theorem}\label{thm:Bound6}
We assume (F${}_d$) for $d\in (0,1/2)$. Let $\delta \in (0,1)$.
Then there exists $M_6 \in (0,\infty)$, which depends only on $w$ and $\delta$, 
such that
\begin{equation}
\sum_{s=1}^{[n\delta]} \Delta_n^{s,t} \le M_6 n^{-d}, \qquad
n\in\N,\ \ t \in \{ 1,\dots, n \}.
\label{eq:col-sum4}
\end{equation}  
\end{theorem}
The proof of Theorem \ref{thm:Bound6} will also be given in Section \ref{sec:8}.


\section{Applications} \label{sec:appl} 

\subsection{Baxter-type theorems} \label{sec:baxter}

In this section, we present two Baxter-type theorems for the vector ARFIMA model with $d\in (0,1/2)$. 

Before giving the results, we first review Baxter's inequality, which dates back to 
the work of Baxter \cite{p:bax-62} in 1962. 
Let
$\{\phi_{k}\}_{k\in\N}$ and $\{\phi_{n,k}\}_{k=1}^{n}$ be the forward infinite and finite predictor coefficients of a second-order stationary process $\{X_k\}$ with spectral density $w$, respectively (cf.\ \cite{b:bro-dav-06}, Section 11.4). Then, 
Baxter's inequality, which was proved under (\ref{eq:S}) in \cite{p:bax-62, CP, HD} 
and under ($\mathrm{F}_d$) with $d\in (0,1/2)$ in \cite{p:ino-kas-06, IKP2},
asserts that, there exists $K\in (0,\infty)$ such that
\begin{equation}
\sum_{k=1}^{n}\Vert \phi_{n,k} - \phi_k \Vert
\le K\sum_{k=n+1}^{\infty}\Vert \phi_k\Vert,\quad
n\in\N.
\label{eq:Baxter597}
\end{equation}
In particular, we have
\begin{equation}
\lim_{n\to\infty} \sum_{k=1}^{n} \Vert \phi_{n,k} -\phi_k\Vert = 0.
\label{eq:pred-234}
\end{equation}
The convergence result (\ref{eq:pred-234}) has various applications in time series analysis, such as 
the autoregressive sieve bootstrap (see, e.g., \cite{IKP2} and the references therein). 

Now, let the output sequence $\mathbf{y}$ belong to either $\mathcal{A}_{\rho}^{\infty \times q}$ with  $\rho \in (1-2d, \infty)$ or 
$\mathcal{B}^{\infty \times q}$, where, for $\rho \in (0,\infty)$,
\begin{align}
\mathcal{A}_{\rho}^{\infty \times q} 
&:= \left\{ \mathbf{y} = (y_1^\top, y_2^\top, \dots)^\top \in \C^{\infty\times q}: 
\sup_{k \in \N}  k^{\rho} \|y_k\|   <\infty \right\},
\label{eq:Arho123}\\
\mathcal{B}^{\infty \times q} &
:= \left\{ \mathbf{y} = (y_1^\top, y_2^\top, \dots)^\top \in \C^{\infty\times q}: 
\sum_{k=1}^{\infty} \|y_k\| <\infty \right\}.
\label{eq:B123}
\end{align}
For $n \in \N$, let $\mathbf{z} = (z_1^\top, z_2^\top, \dots)^\top$ 
and $\mathbf{z}_n = (z_{1,n}^\top, \dots, z_{n,n}^\top)^\top$ be as in (\ref{eq:tsysinf123}) and (\ref{eq:tsysn123}) with (\ref{eq:yn234}), respectively.

Here are our two Baxter-type theorems for long memory processes.

\begin{theorem} \label{thm:B-type1}
We assume (F${}_d$) for $d\in (0,1/2)$. 
Let $\mathbf{y} \in \mathcal{A}_{\rho}^{\infty \times q}$ 
with $\rho \in (1-2d,\infty)$.  Then, 
\begin{equation}
\sum_{k=1}^{n} \|z_{k} - z_{k,n}\| =
\begin{cases}
O(n^{1-2d-\rho}) &(1-2d<\rho<1-d),\\
O(n^{-d}\log n) &(\rho=1-d),\\
O(n^{-d}) &(\rho>1-d),
\end{cases}
\quad n\to\infty.
\label{eq:BaxterA}
\end{equation}
In particular, (\ref{eq:Baxter}) holds.
\end{theorem}

\begin{theorem} \label{thm:B-type2}
We assume (F${}_d$) for $d\in (0,1/2)$. 
Let $\mathbf{y} \in \mathcal{B}^{\infty \times q}$ and $\kappa \in (0,d/(1-d))$.
Then,
\begin{equation} \label{eq:BaxterB}
\sum_{k=1}^{n} \|z_{k} - z_{k,n}\| = O(n^{-d+\kappa(1-d)} ) 
+ O\bigg( \sum_{t=[n^{\kappa}]+1}^{\infty}  \|y_{t}\| \bigg),
\quad n \to \infty.
\end{equation}
In particular, (\ref{eq:Baxter}) holds.
\end{theorem}
The proofs of Theorems \ref{thm:B-type1} and \ref{thm:B-type2} will be given in Section \ref{sec:9}.

Now, we examine the relationship between Baxter's inequality as in (\ref{eq:pred-234}) and 
our Baxter-type theorems. For this purpose, we define
\begin{equation} \label{eq:wtilde}
\tilde{w}(e^{i\theta}):=w(e^{-i\theta}).
\end{equation} Then, 
$\tilde{w}$ is the spectral density of 
the time-reversed process $\{\tilde{X}_k\}$ defined by 
$\tilde{X}_k := X_{-k}$ for $k \in \Z$, hence 
$\tilde{w}$ satisfies ($\mathrm{F}_d$) if and only if $w$ does as well for the same $d$. See \cite{IKP2}, Section 6.2.

It is well-known that the forward infinite and finite predictor coefficients are embedded in the first columns of $T_\infty(\tilde{w})^{-1}$ and $T_{n+1}(\tilde{w})^{-1}$, respectively. That is, the corresponding Wiener--Hopf solutions $\mathbf{z} = T_\infty(\tilde{w})^{-1} \mathbf{y}$ and $\mathbf{z}_{n+1} = T_{n+1}(\tilde{w})^{-1} \mathbf{y}_{n+1}$ for $\mathbf{y} = (I_q, O_q, O_q, \dots)^\top$, where $I_q$ and $O_q$ denote the $q\times q$ identity matrix and zero matrix, are respectively given by
\begin{equation}\label{eq:zandzn}
\mathbf{z}^\top =  
( v_{\infty}^{-1}, v_{\infty}^{-1} \phi_{1}, v_{\infty}^{-1} \phi_{2}, \dots)
\quad \text{and} \quad 
\mathbf{z}_{n+1}^\top = 
( v_{n+1}^{-1}, v_{n+1}^{-1} \phi_{n,1}, \dots, v_{n+1}^{-1} \phi_{n,n} ).
\end{equation}
Here, $v_\infty \in \C^{q\times q}$ and $v_{n+1} \in \C^{q\times q}$ are the infinite and finite prediction error covariances, respectively (see e.g., \cite{IKP2}, equations (5.3) and (6.16)). 
Using Theorem \ref{thm:B-type1} and the above expressions, 
we can give an alternative proof of Baxter's inequality under (F${}_d$) with $d\in (0,1/2)$.

\begin{corollary} \label{cor:baxter}
We assume (F${}_d$) for $d\in (0,1/2)$. Then, (\ref{eq:pred-234}) holds.
\end{corollary}

\begin{proof}
As we mentioned above, $\tilde{w}$ satisfies (F${}_d$) for the same $d$ as $w$. 
First, we note that $\mathbf{y} = (I_q, O_q, O_q, \dots)^\top \in \mathcal{A}_{\rho}^{\infty \times q}$ 
for any $\rho \in (0,\infty)$. Therefore, by applying Theorem \ref{thm:B-type1} to $T_{n+1}(\tilde{w})$ 
and using (\ref{eq:zandzn}), we have
\begin{equation} \label{eq:Vsum1}
\|v_{n+1}^{-1} - v_{\infty}^{-1} \| + \sum_{k=1}^{n} \|v_{n+1}^{-1} \phi_{n,k}- v_{\infty}^{-1} \phi_{k}\| 
= O(n^{-d}),
\quad n \to \infty.
\end{equation}
Next, by using an expansion similar to that in \cite{p:kre-11}, page 2124, we have
\begin{equation}  \label{eq:Vsum2}
\sum_{k=1}^{n} \|\phi_{n,k} - \phi_{k}\|
\leq \|v_{n+1}\| \sum_{k=1}^{n} \|v_{n+1}^{-1} \phi_{n,k}- v_{\infty}^{-1} \phi_{k}\|
+ \|v_{n+1}\| \|v_{n+1}^{-1}-v_{\infty}^{-1}\| \sum_{k=1}^{n} \| \phi_{k}\|.
\end{equation}
We note that the sequence $\{\|v_{n+1}\|\}$ is bounded (cf. \cite{IKP2}, page 1227) and 
that $\sum_{k=1}^{\infty} \| \phi_{k}\| <\infty$ holds by \cite{IKP2}, (2.16). 
Therefore, substituting (\ref{eq:Vsum1}) into (\ref{eq:Vsum2}), we get (\ref{eq:pred-234}).
\end{proof}

Note, however, that our results do not generalize to Baxter's inequality for multi-step predictor coefficients. In this case, alternative approaches of \cite{p:yang-24} are necessary.

\subsection{Alternative approximations of inverse Toeplitz matrices} \label{sec:alt}

To assess the accuracy of the approximations of inverse Toeplitz matrices, we first introduce some norms for block matrices.
Let $A_n \in \C^{qn \times qn}$ with $(A_n)^{s,t}\in \C^{q\times q}$ for $s,t \in \{1, \dots, n\}$. 
For $p \in [1,\infty]$, the block $\ell_p$-norm of $A_n$, denoted by $\|A_n\|_{p,\text{block}}^{q \times q}$, is defined as the usual  $\ell_p$ operator norm of the matrix $(\| (A_n)^{s,t}\|)_{1\leq s,t \leq n} \in \R^{n \times n}$. For example,
\begin{equation*}
\|A_n\|_{1,\text{block}}^{q \times q} := \sup_{1\leq t\leq n} \sum_{s=1}^{n} \|(A_n)^{s,t}\| \quad \text{and} \quad
\|A_n\|_{\infty,\text{block}}^{q \times q} := \sup_{1\leq s\leq n} \sum_{t=1}^{n} \|(A_n)^{s,t}\|.
\end{equation*}

For $n \in \N$, 
let $[T_\infty(w)^{-1}]_n:= ((T_\infty(w)^{-1})^{s,t})_{1\leq s,t \leq n} \in \C^{nq \times nq}$ 
be the truncation of $T_\infty(w)^{-1}$. Then, by Lemma \ref{lemma:CRequal333} and Theorem \ref{thm:Bound5}, 
under (F${}_d$) for $d\in (0,1/2)$, we have,
\begin{equation} \label{eq:block-norm1}
\text{for  $p \in \{1,\infty\}$,} \quad 
\| T_n(w)^{-1} - [T_\infty(w)^{-1}]_n\|_{p,\text{block}}^{q \times q} = O(1), \qquad n \to \infty.
\end{equation}
Inspecting Theorems \ref{thm:Bound2}, \ref{thm:Bound4} and \ref{thm:Bound5} (also, the numerical results displayed in the left panel of Figure \ref{fig:1}), we see that 
the $O(1)$ bound above is due to the large value of $\Delta_n^{s,t}$ 
when both $s$ and $t$ are close to $n$. To ameliorate this issue, we replace the right bottom corner of $[T_\infty(w)^{-1}]_n$ with the  ``superior'' elements. One key idea is to focus on the following identity in Proposition \ref{prop:time-reverse}:
\begin{equation}\label{eq:TnTtilde}
(T_n(w)^{-1})^{n+1-s,n+1-t} = (T_n(\tilde{w})^{-1})^{s,t}, \qquad n \in\N,~~ s,t \in \{1, \dots, n\}.
\end{equation}
Here $\tilde{w}$ is as in (\ref{eq:wtilde}). 
Thanks to (\ref{eq:TnTtilde}), we can approximate the right bottom corner of $T_n(w)^{-1}$ (which corresponds to the left top corner of $T_n(\tilde{w})^{-1}$) by the left top corner of $T_{\infty}(\tilde{w})^{-1}$. 
More precisely, for $\delta \in (0,1/2]$, we define the block Toeplitz matrix $\Omega_{n,\delta}(w) \in \C^{nq \times nq}$ by
\begin{equation}
(\Omega_{n,\delta}(w))^{s,t}
=
\begin{cases}
(T_{\infty}(w)^{-1})^{s,t}, \quad\qquad\qquad\qquad\qquad\, 
(s,t) \in H_{\delta}(n) \times H_{\delta}(n),  \\
(T_{\infty}(\tilde{w})^{-1})^{n+1-s,n+1-t}, 
\quad\qquad\qquad\, (s,t) \in T_{\delta}(n) \times  T_{\delta}(n), \\
\frac{1}{2} \left\{ (T_{\infty}(w)^{-1})^{s,t} +  (T_{\infty}(\tilde{w})^{-1})^{n+1-s,n+1-t} \right\},
\qquad \text{otherwise},
\end{cases}
\label{eq:Omega}
\end{equation}
where
\begin{equation}
H_{\delta}(n) := \{1, \dots,[\delta n]\}
\quad
\mbox{and}
\quad
T_{\delta}(n) := \{n-[\delta n]+1, \dots,n\}.
\label{eq:AdeltaBdelta}
\end{equation}
We note 
that the entries of $T_{\infty}(w)^{-1}$ and $T_{\infty}(\tilde{w})^{-1}$ are given in (\ref{eq:Tinftyinv973}) and (\ref{eq:Tinftyinvtilde}), respectively. Therefore, $\Omega_{n,\delta}(w)$ can be evaluated in terms of the forward and backward infinite-order autoregressive coefficients. 

The next theorem shows that $\Omega_{n,\delta}(w)$ approximates $T_n(w)^{-1}$ 
better than the truncation $[T_\infty(w)^{-1}]_n$ of $T_\infty(w)^{-1}$.

\begin{theorem} \label{thm:Omega}
We assume (F${}_d$) for $d\in (0,1/2)$. For $\delta \in (0,1/2]$, let $\Omega_{n,\delta}(w) \in \C^{nq \times nq}$ be defined as in (\ref{eq:Omega}). Then, the following two assertions hold:
\begin{itemize}
\item[(i)] $\Omega_{n,\delta}(w)$ is a Hermitian matrix for $n \in \N$.
\item[(ii)] For $p \in \{1,\infty\}$, $\|T_n(w)^{-1} - \Omega_{n,\delta}(w) \|_{p,\emph{block}}^{q\times q} = O(n^{-d})$
as $n \to \infty$.
\end{itemize}
\end{theorem}
The proof of Theorem \ref{thm:Omega} is given below. 

\begin{remark}
Note that $\Omega_{n,\delta}(w)$ may not be a positive definite matrix. 
If positive-definiteness is required, a small ridge term $\lambda_n I_{nq}$ 
can be added to $\Omega_{n,\delta}(w)$, where $\{\lambda_n\}$ is a suitably chosen 
sequence of non-negative numbers. Then $\Omega_{n,\delta}(w) + \lambda_n I_{nq}$ 
becomes positive-definite, while substituting $\Omega_{n,\delta}(w)+\lambda_n I_{nq}$ 
for $\Omega_{n,\delta}(w)$ does not 
change the statement of Theorem \ref{thm:Omega}(ii). 
Details of this work will be investigated in the future research.
\end{remark}


\section{Key upper bound for $\Delta_n^{s,t}$}  \label{sec:4}

In this section, we derive a key upper bound for $\Delta_n^{s,t}$ in (\ref{eq:Delta}) 
(Theorem \ref{thm:Tnbound} below). 
The proofs of the theorems given in Section \ref{sec:3} are based on this bound.
The bound is given in terms of the $q\times q$ matrix sequence 
$\{\beta_k\}$ to be explained below.

We define terms before presenting our results. 
Let $\C^{m\times n}$ be the set of all complex $m\times n$ matrices. We write 
$\C^q$ for $\C^{q\times 1}$. 
The transpose and Hermitian conjugate of $a\in \C^{m\times n}$ 
are denoted by $a^{\top}$ and $a^*$, respectively. 
For $p\in [1,\infty)$, we write $\ell_{p}^{q\times q}$ for the space of all
$\C^{q\times q}$-valued sequences $\{a_k\}_{k=0}^{\infty}$ such that 
$\sum_{k=0}^{\infty} \Vert a_k\Vert^p<\infty$. Let 
$L_2(\T)$ be the space of Lebesgue measurable functions 
$g:\T\to\C$ such that 
$\int_{-\pi}^{\pi}\vert g(e^{i\theta})\vert^2 d\theta< \infty$. 
The Hardy class $H_2(\T)$ on $\T$ is the closed subspace of 
$L_2(\T)$ consisting of $g\in L_2(\T)$ such that
$\int_{-\pi}^{\pi}e^{im\theta} g(e^{i\theta})d\theta=0$ for all $m \in \N$. 
Let $H_2^{q\times q}(\T)$ be the space of $\C^{q\times q}$-valued functions on
$\T$ whose entries belong to $H_2(\T)$. 
Let $\D:=\{z\in\C : \vert z\vert<1\}$ be the open unit disk in $\C$. 
We write $H_2(\D)$ for the Hardy class on $\D$, consisting of 
holomorphic functions $g$ on $\D$ such that
$\sup_{r\in [0,1)}\int_{-\pi}^{\pi}\vert g(re^{i\theta})\vert^2\sigma(d\theta)<\infty$, where $\sigma$ is the normalized Lebesgue measure $d\theta/(2\pi)$ on $[-\pi,\pi)$. 
As usual, we identify each function $g$ in $H_2(\D)$ with its boundary function
$g(e^{i\theta}):=\lim_{r\uparrow 1} g(re^{i\theta})$, $\sigma$-a.e.,
in $H_2(\T)$.
A function $h$ in $H_2^{q\times q}(\T)$ is called an \textit{outer function} if $\det h$ is a
$\C$-valued outer function (cf. \cite{KK}, Definition 3.1).

In what follows in this section, 
we assume that the symbol $w$ in (\ref{eq:gamma123}) satisfies 
(F${}_d$) for $d\in (0,1/2)$. 
For $g$ appearing in (F${}_d$), we take $g_{\sharp} : \T\to \C^{q\times q}$ satisfying the condition (C) and
\begin{equation}
g(e^{i\theta}) g(e^{i\theta})^*=g_{\sharp}(e^{i\theta})^*g_{\sharp}(e^{i\theta}), 
\quad 
\theta \in [-\pi, \pi)
\label{eq:decomp-g666}
\end{equation}
(see \cite{IKP2}, Section 6.2). We define two outer 
functions $h$ and $h_{\sharp}$ in $H_2^{q\times q}(\T)$ by
\begin{equation}
h(z)=(1-z)^{-d}g(z) \quad \mbox{and} \quad 
h_{\sharp}(z)=(1-z)^{-d}g_{\sharp}(z),
\label{eq:ratouter372}
\end{equation}
respectively. 
Then $w$ admits the following two decompositions:
\begin{equation}
w(e^{i\theta}) 
= h(e^{i\theta}) h(e^{i\theta})^*
= h_{\sharp}(e^{i\theta})^*h_{\sharp}(e^{i\theta})
, \quad
\text{$\sigma$-a.e.}
\label{eq:decomp888}
\end{equation}

We define another outer function $\tilde{h}$ in $H_2^{q\times q}(\T)$ by 
$\tilde{h}(z) := \{h_{\sharp}(\overline{z})\}^*$. 
Since both $h^{-1}$ and $\tilde{h}^{-1}$ also 
belong to $H_2^{q\times q}(\T)$, 
we can define two sequences $\{a_k\}$ and $\{\tilde{a}_k\}$ belonging 
to $\ell_{2}^{q\times q}$ by
\begin{equation} \label{eq:Wold}
-h(z)^{-1}=\sum_{k=0}^{\infty}z^ka_k
\quad \text{and} \quad
-\tilde{h}(z)^{-1}=\sum_{k=0}^{\infty}z^k\tilde{a}_k,
\quad z\in\D,
\end{equation}
respectively. 
By Proposition \ref{prop:aA-est123}(i) below, 
both $\{a_k\}$ and $\{\tilde{a}_k\}$ actually belong 
to $\ell_{1}^{q\times q}$. 
We call the coefficients $a_k$'s and $\tilde{a}_k$'s the {\it infinite-order forward and backward autoregressive coefficients}, respectively. Finally, we define a $\C^{q\times q}$-valued 
sequence $\{\tilde{A}_{k}\}_{k=0}^{\infty}$ by
\begin{equation}
\tilde{A}_k := -\sum_{u=k}^{\infty} \tilde{a}_u,\quad k\in \N\cup\{0\}.
\label{eq:tildeA656}
\end{equation}

\begin{proposition}\label{prop:aA-est123}
We assume $(\mathrm{F}_{d})$ for $d\in (0,1/2)$. 
Let $\{a_k\}$, $\{\tilde{a}_k\}$ and $\{\tilde{A}_{k}\}$ be as in (\ref{eq:Wold}) 
and (\ref{eq:tildeA656}). 
Then the following three assertions hold:
\begin{itemize}
\item[(i)] There exists $K_1\in (0,\infty)$ such that 
$\max( \Vert a_n \Vert, \Vert \tilde{a}_n \Vert ) \le K_1(n+1)^{-1-d}$ holds 
for $n\in\N$.
\item[(ii)] There exists $K_2 \in (0,\infty)$ such that 
$\Vert \tilde{A}_n\Vert \le K_2 n^{-d}$ holds for $n\in\N$.
\item[(iii)] $\sum_{u=0}^k \tilde{a}_u = \tilde{A}_{k+1}$ holds for $k\in\N\cup\{0\}$.
\end{itemize}
\end{proposition}

\begin{proof}
The assertion (i) follows from (6.23) and (6.24) in \cite{IKP2}, 
while (ii) from (i). Since 
$\sum_{k=0}^{\infty} \tilde{a}_k = -\tilde{h}(1)^{-1}
= - (h_{\sharp}(1)^{-1})^* =0$, 
we obtain (iii).
\end{proof}

We define a $\C^{q\times q}$-valued sequence $\{\beta_k\}_{k \in \Z}$ as the 
(negative of the) Fourier coefficients of the phase function 
$h^*h_{\sharp}^{-1}=h^{-1}h_{\sharp}^*$ (cf.\cite{IKP2}, Section 4):
\begin{equation}
\beta_k 
=
-\int_{-\pi}^{\pi}e^{-ik\theta} h(e^{i\theta})^* h_{\sharp}(e^{i\theta})^{-1} \frac{d\theta}{2\pi}, 
\quad k \in \Z.
\label{eq:beta-def667}
\end{equation}

For $n\in\mathbb{N}$, $u \in \{1,\dots,n\}$ and $k\in\mathbb{N}$, we can define 
the sequences 
$\{\tilde{b}_{n,u,\ell}^k\}_{\ell=0}^{\infty}\in \ell_{2}^{q\times q}$
by the recursion
\begin{equation}
\left\{
\begin{aligned}
\tilde{b}_{n,u,\ell}^1&=\beta_{n+1-u+\ell}^*,\\
\tilde{b}_{n,u,\ell}^{2k}
&=\sum_{m=0}^{\infty} \tilde{b}_{n,u,m}^{2k-1} \beta_{n+1+m+\ell},\quad
\tilde{b}_{n,u,\ell}^{2k+1}
=\sum_{m=0}^{\infty} \tilde{b}_{n,u,m}^{2k} \beta_{n+1+m+\ell}^*
\end{aligned}
\right.
\label{eq:til-recurs123}
\end{equation}
(cf. \cite{I23}, Section 2). 
Then, for $n \in \N$ and $s,t \in \{1,\dots,n\}$, we have
\begin{equation}
\begin{aligned}
\left(T_n(w)^{-1}\right)^{s,t}  - \left(T_{\infty}(w)^{-1}\right)^{s,t} 
= \sum_{k=1}^{\infty} 
\left\{ \sum_{u=1}^t 
\sum_{\ell = 0}^{\infty} a_{n + 1 - s + \ell}^* (\tilde{b}_{n,u,\ell}^{2k-1})^*   \tilde{a}_{t-u}
+ \sum_{u=1}^t \sum_{\ell = 0}^{\infty} \tilde{a}_{s + \ell}^* (\tilde{b}_{n,u,\ell}^{2k})^*  \tilde{a}_{t-u}\right\}
\end{aligned}
\label{eq:key-equality}
\end{equation}
(see \cite{I23}, Theorem 2.1 and Remark 2.1). 
Formula (\ref{eq:key-equality}) plays an important role below.

\begin{proposition}\label{prop:sum-by-parts134}
We assume (F${}_d$) for $d\in (0,1/2)$. Then, for $n, k\in\N$ and $t, \ell\in\N\cup\{0\}$, we have
\begin{equation*}
\sum_{u=0}^{t} \tilde{a}_u^* \tilde{b}^k_{n,t+1-u,\ell}
=\tilde{A}_{t+1}^* \tilde{b}^k_{n,1,\ell} 
- \sum_{u=0}^{t-1} \tilde{A}_{u+1}^* ( \tilde{b}^k_{n,t-u,\ell} - \tilde{b}^k_{n,t+1-u,\ell}),
\end{equation*}
where the convention $\sum_{u=0}^{-1} = 0$ is adopted in the sum on the right-hand side.
\end{proposition}

\begin{proof}
By summation by parts and Proposition \ref{prop:aA-est123} (iii), we have
\[
\begin{aligned}
\sum_{u=0}^t \tilde{a}_u^* \tilde{b}^k_{n,t+1-u,\ell}
& = \left( \sum_{u=0}^t \tilde{a}_u^* \right) \tilde{b}^k_{n,1,\ell} 
- \sum_{u=0}^{t-1} \left( \sum_{r=0}^u \tilde{a}_r^* \right) ( \tilde{b}^k_{n,t-u,\ell} - \tilde{b}^k_{n,t+1-u,\ell})\\
&=\tilde{A}_{t+1}^* \tilde{b}^k_{n,1,\ell} 
- \sum_{u=0}^{t-1} \tilde{A}_{u+1}^* ( \tilde{b}^k_{n,t-u,\ell} - \tilde{b}^k_{n,t+1-u,\ell}).
\end{aligned}
\]
Thus the proposition follows.
\end{proof}

For $k,n\in \N$ and $s, t \in \{0,\dots,n-1\}$, we define
\begin{align}
S_{1,k}(n,s,t) 
&:=\sum_{\ell = 0}^{\infty} 
\frac{1}{(s + \ell +2)^{1 + d}(t+1)^d} \Vert \tilde{b}^k_{n,1,\ell} \Vert, 
\label{eq:S1k}\\
S_{2,k}(n,s,t) 
&:=  \sum_{\ell = 0}^{\infty} \sum_{u=0}^{t-1} 
\frac{1}{(s + \ell +2)^{1+d}(u + 1)^d} 
\Vert \tilde{b}^k_{n,t-u,\ell} - \tilde{b}^k_{n,t+1-u,\ell} \Vert.
\label{eq:S2k}
\end{align}

Here is the key upper bound for $\Delta_n^{s,t}$. 

\begin{theorem} \label{thm:Tnbound}
We assume (F${}_d$) for $d\in (0,1/2)$. Let $K_1$ and $K_2$ be as in 
Proposition \ref{prop:aA-est123}. 
Then, for $n \in \N$ and $s,t \in \{1, \dots,n\}$, we have
\begin{equation}
\begin{aligned}
\Delta_n^{s,t}
&\leq K_1 K_2 \sum_{k=1}^{\infty} \{ S_{1,2k-1}(n,n-s,t-1) + S_{2,2k-1}(n,n-s,t-1)\}
\\
& \quad\quad +
K_1 K_2 \sum_{k=1}^{\infty} \{ S_{1,2k}(n,s-1,t-1) + S_{2,2k}(n,s-1,t-1)\}.
\end{aligned}
\label{eq:key-estimate}
\end{equation}
\end{theorem}

\begin{proof}
By (\ref{eq:key-equality}), 
$\Delta_n^{s,t} \le \sum_{k=1}^{\infty} \{ D_k(n,s,t) + E_k(n,s,t) \}$, 
where
\begin{align*}
D_k(n,s,t)
&:=\sum_{\ell = 0}^{\infty}
\Vert a_{n + 1 - s + \ell} \Vert \left\Vert \sum_{u=1}^t (\tilde{b}_{n,u,\ell}^{2k-1})^* \tilde{a}_{t-u} 
\right\Vert,\quad k\in\N,\\
E_k(n,s,t)
&:=
\sum_{\ell = 0}^{\infty}
\Vert \tilde{a}_{s + \ell} \Vert \left\Vert \sum_{u=1}^t (\tilde{b}_{n,u,\ell}^{2k})^* \tilde{a}_{t-u}
\right\Vert,\quad k\in\N.
\end{align*}
Using Proposition \ref{prop:aA-est123}(i) and the change of variable $u^\prime = t-u$, 
we have
\[
D_k(n,s,t) \leq
 K_1 \sum_{\ell = 0}^{\infty}
\frac{\Vert \sum_{u=1}^t \tilde{a}_{t-u}^* \tilde{b}_{n,u,\ell}^{2k-1} \Vert}{(n + 2 - s + \ell)^{1+d}}
=
K_1 \sum_{\ell = 0}^{\infty}
\frac{\Vert \sum_{u=0}^{t-1} \tilde{a}_{u}^* \tilde{b}_{n,(t-1)+1-u,\ell}^{2k-1}\Vert}{(n - s + \ell + 2)^{1+d}}.
\]
Using the reparameterizations $s^\prime = n-s$ and $t^\prime = t-1$ and 
Proposition \ref{prop:sum-by-parts134}, 
\[
\begin{aligned}
D_k(n,s,t)&\leq 
 K_1 \sum_{\ell = 0}^{\infty}
\frac{\Vert \sum_{u=0}^{t^\prime} \tilde{a}_{u}^* \tilde{b}_{n,t^\prime+1-u,\ell}^{2k-1} \Vert}{(s^\prime + 2 + \ell)^{1+d}} \\
&\leq
 K_1 \sum_{\ell = 0}^{\infty} 
\frac{\Vert \tilde{A}_{t^\prime+1} \Vert \Vert \tilde{b}^{2k-1}_{n,1,\ell} \Vert}{(s^\prime + \ell +2)^{1+d}} + K_1 \sum_{u=0}^{t^\prime-1} 
\frac{\Vert \tilde{A}_{u+1}\Vert \Vert \tilde{b}^{2k-1}_{n,t^\prime-u,\ell} - \tilde{b}^{2k-1}_{n,t^\prime+1-u,\ell} \Vert}{(s^\prime + \ell +2)^{1+d}} 
 \\
&\leq K_1 K_2 \left\{S_{1,2k-1}(n,s^\prime,t^\prime) + S_{2,2k-1}(n,s^\prime,t^\prime)\right\} \\
&= K_1 K_2 \left\{S_{1,2k-1}(n,n-s,t-1) + S_{2,2k-1}(n,n-s,t-1)\right\}.
\end{aligned}
\]
Similarly, 
$E_k(n,s,t) \leq K_1 K_2 \left\{S_{1,2k}(n,s-1,t-1) + S_{2,2k}(n,s-1,t-1)\right\}$. 
Combining, we obtain the desired result.
\end{proof}


\section{Bounds for sums consisting of $\beta_n$} \label{sec:5}

In this section, we prove some bounds for sums consisting of $\beta_n$'s. 
These bounds will be used to prove Theorems \ref{thm:Bound1} and \ref{thm:Bound3}.

\begin{proposition}\label{prop:beta123}
We assume (F${}_d$) for $d\in (0,1/2)$. 
\begin{enumerate}
\item There exists $K_3\in (0,\infty)$ such that the following two 
inequalities hold:
\begin{align}
\Vert \beta_n \Vert &\le K_3 (n+1)^{-1}, \quad n\in\N,
\label{eq:betaestim216}\\
\Vert \beta_{n+1} - \beta_n \Vert &\le K_3 (n+2)^{-2}, \quad n\in\N.
\label{eq:betaestim233}
\end{align}
\item For any $r\in (1,\infty)$, there exists $N_1\in\N$ such that
\begin{equation}
\Vert \beta_n \Vert \le \frac{r \sin(\pi d)}{\pi (n+1)}, \quad n\in \{ N_1, N_1+1, \dots \}.
\label{eq:betaestim327}
\end{equation}
\end{enumerate}
\end{proposition}

\begin{proof}
For $n\in\N$, $\rho_n:=\sin(\pi d)/\{\pi(n-d)\}$, and 
the unitary matrix $U:=g(1)^* g_{\sharp}(1)^{-1}\in \C^{q\times q}$, 
we define $D_n\in\C^{q\times q}$ by 
$\beta_n = \rho_n(I_q + D_n)U$. 
Then, by \cite{IKP2}, Proposition 6.4, there exists $\kappa\in (0,\infty)$ such that 
$\Vert D_n\Vert \le \kappa n^{-1}$ holds for $n\in\N$. 
Therefore, (\ref{eq:betaestim216}) and (\ref{eq:betaestim327}) hold 
for some $K_3\in (0,\infty)$ and $N_1\in\N$, respectively. 
Moreover, since
\[
\begin{aligned}
\Vert \beta_{n+1} - \beta_n \Vert
&= \Vert (\rho_{n+1} - \rho_n)U + \rho_{n+1} D_{n+1}U - \rho_n D_n U \Vert\\
&\le \vert \rho_{n+1} - \rho_n \vert + \vert\rho_{n+1}\vert \Vert D_{n+1}\Vert 
+ \vert \rho_n\vert \Vert D_n \Vert
\end{aligned}
\]
and $\rho_{n} - \rho_{n+1} = O(n^{-2})$ as $n\to\infty$, 
we obtain $\Vert \beta_{n+1} - \beta_n \Vert = O(n^{-2})$ as $n\to\infty$, 
hence (\ref{eq:betaestim233}) by replacing $K_3$ if necessary.
\end{proof}

For $m_1,\ell\in (0,\infty)$, we define
\begin{align}
&D_2(m_1,\ell) := \frac{1}{m_1+\ell},
\label{eq:D2}\\
&
D_k(m_1,\ell) 
:= \int_{0}^{\infty}dm_2 \cdots \int_{0}^{\infty} dm_{k-1} 
\frac{1}{\{\prod_{i=1}^{k-2} (m_{i}+m_{i+1})\}(m_{k-1} + \ell)},\ \ 
k\ge 3.
\label{eq:D3more}
\end{align}

\begin{proposition}\label{prop:btilineq111}
We assume (F${}_d$) for $d\in (0,1/2)$. Take $r\in (1,\infty)$, and let $K_3\in (0,\infty)$ and 
$N_1\in\N$ be as in Proposition \ref{prop:beta123}. 
Then, 
for $\ell\in (0,\infty)$, 
the following inequalities hold:
\begin{align}
&\Vert \tilde{b}^1_{n,1,[n \ell]} \Vert 
\le n^{-1}K_3 \frac{1}{1 + \ell},\quad n\in\N,
\label{eq:btilineq111}\\
&
\Vert \tilde{b}^k_{n,1,[n \ell]} \Vert 
\le n^{-1}K_3 \left( \frac{r\sin(\pi d)}{\pi} \right)^{k-1}
\int_0^{\infty} \frac{D_k(m_1,\ell)}{1+m_1}  dm_1, \ \ 
n\ge N_1, \ k\ge 2.
\label{eq:btilineq222}
\end{align}
\end{proposition}

\begin{proof}
We only prove (\ref{eq:btilineq222}) for $k = 3$; the other cases can be treated in the same
way. For $\ell\in (0,\infty)$ and $n\in \{N_1, N_1+1, \dots\}$, 
\[
\begin{aligned}
&(K_3)^{-1} \left( \frac{r\sin(\pi d)}{\pi} \right)^{-2} 
\Vert \tilde{b}^3_{n,1,[n \ell]} \Vert \\
&\le 
(K_3)^{-1} \left( \frac{r\sin(\pi d)}{\pi} \right)^{-2} 
\sum_{m_1=0}^{\infty} \sum_{m_2=0}^{\infty} 
\Vert \beta_{n+m_1} \Vert \Vert \beta_{n+1+m_1+m_2} \Vert \Vert \beta_{n+1+m_2+[n\ell]} \Vert\\
&\le 
\sum_{m_1=0}^{\infty} \sum_{m_2=0}^{\infty} 
\frac{1}{(n+m_1+1) (n+m_1+m_2+2) (n+m_2+[n\ell]+2)}\\
&\le 
\sum_{m_1=0}^{\infty} \sum_{m_2=0}^{\infty} 
\frac{1}{(n+m_1+1) (m_1+m_2+2) (m_2+[n\ell]+2)}.
\end{aligned}
\]
Here, the second inequality is by Proposition \ref{prop:beta123}. Therefore,

\[
\begin{aligned}
&(K_3)^{-1} \left( \frac{r\sin(\pi d)}{\pi} \right)^{-2} 
\Vert \tilde{b}^3_{n,1,[n \ell]} \Vert \\
&\le 
\sum_{m_1=0}^{\infty} \sum_{m_2=0}^{\infty} 
\frac{1}{(n+m_1+1) (m_1+m_2+2) (m_2+[n\ell]+2)} \\
&=\int_{0}^{\infty} dm_1 \int_{0}^{\infty} dm_2 
\frac{1}{(n + [m_1]+1) ([m_1] + [m_2] + 2) ([m_2] + [n\ell]+2)}\\
&= 
\int_{0}^{\infty} dm_1 \int_{0}^{\infty} dm_2 
\frac{n^2}{(n + [n m_1]+1) ([n m_1] + [n m_2] + 2) ([n m_2] + [n\ell]+2)}\\
&\le 
\int_{0}^{\infty} dm_1 \int_{0}^{\infty} dm_2 
\frac{n^{-1}}{(1 + m_1) (m_1 + m_2) (m_2 + \ell)}
=\int_0^{\infty} \frac{n^{-1}}{1+m_1} D_3(m_1,\ell) dm_1.
\end{aligned}
\]
The second equality is by the changes of variables $m_1' = n m_1$ and $m_2' = n m_2$, 
and the last inequality is by $n + [n m_1]+1 > n (1+m_1)$, 
$[n m_1] + [n m_2] + 2 > n (m_1 + m_2)$ and $[n m_2] + [n\ell]+2 > n(m_2 + \ell)$. Thus, we show (\ref{eq:btilineq222}) for $k = 3$.
\end{proof}

\begin{proposition}\label{prop:btilineq222}
We assume (F${}_d$) for $d\in (0,1/2)$. Take $r\in (1,\infty)$, and let $K_3\in (0,\infty)$ and 
$N_1\in\N$ be as in Proposition \ref{prop:beta123}. Then, 
for $t\in (0,1)$, $u\in (0,t)$, $\ell\in (0,\infty)$, and $n\in\{N_1,N_1+1,\dots\}$,
the following inequalities hold:
\begin{align}
&\Vert \tilde{b}^1_{n,[n t] - [n u],[n \ell]} - \tilde{b}^1_{n,[n t] + 1 - [n u],[n \ell]} \Vert 
\le n^{-2}K_3 (1 - t + u + \ell)^{-2},
\label{eq:btilineq333}\\
&
\begin{aligned}
&\Vert \tilde{b}^k_{n,[n t] - [n u],[n \ell]} - \tilde{b}^k_{n,[n t] + 1 - [n u],[n \ell]}\Vert\\&\quad\le n^{-2}K_3 \left( \frac{r\sin(\pi d)}{\pi} \right)^{k-1} 
\int_0^{\infty} \frac{D_k(m_1,\ell)}{(1 - t + u + m_1)^2}  dm_1,\quad k\ge 2.
\end{aligned}
\label{eq:btilineq444}
\end{align}
\end{proposition}

\begin{proof}
We only prove (\ref{eq:btilineq444}) for $k = 3$; the other cases can be treated in the same
way. For $t\in (0,1)$, $u\in (0,t)$, $\ell\in (0,\infty)$, and $n\in\{N_1,N_1+1,\dots\}$, 
in the same way as the proof of Proposition \ref{prop:btilineq111}, 
\[
\begin{aligned}
&(K_3)^{-1} \left( \frac{r\sin(\pi d)}{\pi} \right)^{-2} 
\Vert \tilde{b}^3_{n,[p t] - [p u],[n \ell]} - \tilde{b}^3_{n,[n t] + 1 - [n u],[n \ell]} \Vert \\
&\le 
(K_3)^{-1} \left( \frac{r\sin(\pi d)}{\pi} \right)^{-2} 
\sum_{m_1=0}^{\infty} \sum_{m_2=0}^{\infty} 
\Vert \beta_{n + 1 - [p t] + [p u] + m_1} - \beta_{n - [n t] + [n u] + m_1} \Vert \\
&\quad\quad\quad\quad\quad\quad\quad\quad
\times \Vert \beta_{n+1+m_1+m_2} \Vert \Vert \beta_{n+1+m_2+[n\ell]} \Vert\\
&\le 
\sum_{m_1=0}^{\infty} \sum_{m_2=0}^{\infty} 
\frac{1}{(n - [n t] + [n u] + m_1 + 2)^2 (n+m_1+m_2+2) (n+m_2+[n\ell]+2)}\\
&\le 
\sum_{m_1=0}^{\infty} \sum_{m_2=0}^{\infty} 
\frac{1}{ (n - [n t] + [n u] + m_1 + 2)^2(m_1+m_2+2) (m_2+[n\ell]+2)}\\
&\le 
\int_{0}^{\infty} dm_1 \int_{0}^{\infty} dm_2 
\frac{n^{-2}}{(1 - t + u + m_1)^2 (m_1 + m_2) (m_2 + \ell)}
= \int_0^{\infty} \frac{n^{-2}}{(1 - t + u + m_1)^2} D_3(m_1,\ell) dm_1.
\end{aligned}
\]
Thus, (\ref{eq:btilineq444}) holds for $k = 3$.
\end{proof}

\begin{proposition}\label{prop:btilineq223}
We assume (F${}_d$) for $d\in (0,1/2)$. Take $r\in (1,\infty)$, and let $K_3\in (0,\infty)$ and 
$N_1\in\N$ be as in Proposition \ref{prop:beta123}. Then, 
for $t\in (0,1)$, $u\in (0,t)$, $\ell\in (0,\infty)$, and $n\in\{N_1,N_1+1,\dots\}$,
the following inequalities hold:
\begin{align}
&\Vert \tilde{b}^1_{n,t - [u],[n \ell]} - \tilde{b}^1_{n,t + 1 - [u],[n \ell]} \Vert 
\le n^{-2}K_3 \{1 - (t/n) + (u/n) + \ell\}^{-2},
\label{eq:best171}\\
&
\begin{aligned}
&\Vert \tilde{b}^k_{n,t - [u],[n \ell]} - \tilde{b}^k_{n,t + 1 - [u],[n \ell]} \Vert \\
&\quad \le n^{-2} K_3 \left( \frac{r\sin(\pi d)}{\pi} \right)^{k-1}
\int_0^{\infty} \frac{D_k(m_1,\ell)}{\{1 - (t/n) + (u/n) + m_1\}^2}  dm_1,\quad k\ge 2.
\end{aligned}
\label{eq:best172}
\end{align}
\end{proposition}

The proof is similar to that of Proposition \ref{prop:btilineq222} so is omitted.


\section{Proof of Theorem \ref{thm:Bound1}}\label{sec:6}

In this and next sections, we use the real Hilbert space
\begin{equation}
L := L^2((0,\infty),dx)
\label{eq:L2def234}
\end{equation}
with inner product $(f_1,f_2)_L=\int_0^{\infty} f_1(x)f_2(x)dx$ and 
norm $\Vert f\Vert_L=(f,f)_L^{1/2}$. 
Let $T:L \to L$ be the bounded linear operator defined by
\begin{equation} \label{eq:Tfx}
Tf(x)=\int_0^{\infty} \frac{1}{x+u} f(u) du, \quad f\in L.
\end{equation}
By \cite{p:har-59}, Theorems 315 and 316, the operator norm $\Vert T\Vert$ of $T$ is given by
\begin{equation}
\Vert T\Vert = \pi.
\label{eq:opnormT956}
\end{equation}
By the Fubini--Tonelli theorem, we have, for non-negative $f\in L$,
\begin{equation}
T^{k-1}f(x) = \int_0^{\infty} D_k(x,y) f(y) dy,
\quad k\ge 2.
\label{eq:Dintegral111}
\end{equation}

Recall $S_{1,k}(n,s,t)$ and $S_{2,k}(n,s,t)$ from (\ref{eq:S1k}) and (\ref{eq:S2k}), 
respectively.

\begin{lemma}\label{lemma:Sestim321a}
We assume (F${}_d$) for $d\in (0,1/2)$. 
Take $r\in (1,\infty)$, and let $K_3\in (0,\infty)$ and $N_1\in \N$ be as in 
Proposition \ref{prop:beta123}. Then, for $k\in\N$, $\ n\in \{N_1, N_1+1,\dots\}$ 
and $s\in \{0,\dots, n-1\}$, we have
\begin{equation}
\sum_{t=0}^{n-1} \{ S_{1,k}(n,s,t) + S_{2,k}(n,s,t) \}  \\
\leq
\frac{\pi K_3 \{r\sin(\pi d)\}^{k-1} n^{(1/2)-d}}{(1-4d^2)^{1/2}(s+1)^{(1/2)+d}}.
\label{eq:sumT}
\end{equation}
\end{lemma}

\begin{proof}
In this proof, 
we use an argument similar to that in the proof of Theorem 6.4 of \cite{I00}; 
see also the proof of Proposition 5.3 of \cite{I08}. 
We prove (\ref{eq:sumT}) only for $k\ge 2$; the case $k=1$ can be treated in 
the same way.

Let $n\in \{N_1, N_1+1, \dots\}$ and $s \in \{0,\dots, n-1\}$. 
In the same way as the proof of Proposition \ref{prop:btilineq111}, 
$\sum_{t=0}^{n-1} S_{1,k}(n,s,t)$ is equal to
\[
\begin{aligned}
\sum_{t=0}^{n-1} \sum_{\ell = 0}^{\infty} 
\frac{\Vert \tilde{b}^k_{n,1,\ell} \Vert}{(s + \ell +2)^{1+d}(t+1)^d} 
&= 
\int_0^{n} dt \int_0^{\infty} d\ell 
\frac{\Vert \tilde{b}^k_{n,1,[\ell]} \Vert}{([\ell] + s+ 2)^{1+d}([t] + 1)^d} \\
&=
n^2 \int_0^{1} dt \int_0^{\infty} d\ell 
\frac{\Vert \tilde{b}^k_{n,1,[n \ell]} \Vert}{( [n \ell]+ s+ 2)^{1+d}([n t] + 1)^d} .
\end{aligned}
\]
We define $f_{n,s}\in L$ by 
$f_{n,s}(x) := \{x+(s+1)/n\}^{-1-d}$. 
Its norm is given by
\begin{equation}
\|f_{n,s}\|_{L}^2 = \int_{0}^{\infty} \frac{d\ell}{\{\ell + (s+1)/n\}^{2+2d}}
= \frac{1}{1+2d} \left(\frac{n}{s+1}\right)^{1+2d}.
\label{eq:normf123}
\end{equation}
We also define $\zeta\in L$ by 
$\zeta(x) := (1+x)^{-1}$. 
Then, since $( [n \ell]+ s+ 2)^{-(1+d)} \leq n^{-(1+d)} f_{n,s}(\ell)$ and 
$([n t] + 1)^{-d} \leq (nt)^{-d}$, 
we have
\[
\begin{aligned}
&\sum_{t=0}^{n-1} S_{1,k}(n,s,t)
\le 
n^{-2d} \int_0^{1} \frac{dt}{t^d} \int_0^{\infty} f_{n,s}(\ell)  (n \Vert \tilde{b}^k_{n,1,[n \ell]} \Vert)  d\ell \\
&\le
n^{-2d} \frac{K_3}{1-d} \left( \frac{r\sin(\pi d)}{\pi} \right)^{k-1}
\int_0^{\infty} \frac{1}{1+m_1} \left\{ \int_0^{\infty} D_k(m_1,\ell) f_{n,s}(\ell) d\ell \right\} dm_1 \\
&= n^{-2d} \frac{K_3}{1-d} \left( \frac{r\sin(\pi d)}{\pi} \right)^{k-1} 
(\zeta_{}, T^{k-1} f_{n,s})_L,
\end{aligned}
\]
where the last inequality is by Proposition \ref{prop:btilineq111} 
and the last equality is by (\ref{eq:Dintegral111}).

For the sum of $S_{2,k}(n,s,t)$, we have
\[
\begin{aligned}
&\sum_{t=0}^{n-1} S_{2,k}(n,s,t) 
=
\int_0^{n} dt \int_0^{[t]} du \int_0^{\infty} d\ell 
\frac{\Vert \tilde{b}^k_{n,[t]-[u],[\ell]} - \tilde{b}^k_{n,[t] + 1 - [u],[\ell]} \Vert}{( [\ell] +s+2)^{1+d}([u] + 1)^d} 
\\
&=
n^3 \int_0^{1} dt \int_0^{[nt]/n} du \int_0^{\infty} d\ell 
\frac{\Vert \tilde{b}^k_{n,[n t]-[n u],[n \ell]} - \tilde{b}^k_{n,[n t] + 1 - [n u],[n \ell]} \Vert}{([n \ell]+ s+2)^{1+d}([n u] + 1)^d}\\
&\le 
n^{-2d}  \int_0^{1} dt \int_0^{t} \frac{du}{u^d} \int_0^{\infty} 
f_{n,s}(\ell) (n^2 \Vert \tilde{b}^k_{n,[n t]-[n u],[n \ell]} - \tilde{b}^k_{n,[n t] + 1 - [n u],[n \ell]} \Vert )  d\ell
\\
&=
n^{-2d} \int_0^{1}  \frac{du}{u^d}  \int_u^{1} dt \int_0^{\infty} f_{n,s}(\ell) (n^2 \Vert \tilde{b}^k_{n,[nt]-[n u],[n \ell]} - \tilde{b}^k_{n,[n t] + 1 - [n u],[n \ell]} \Vert )  d\ell.
\end{aligned}
\]
Then, using Proposition \ref{prop:btilineq222} and (\ref{eq:Dintegral111}), 
we obtain
\[
\begin{aligned}
& (K_3)^{-1} \left( \frac{r\sin(\pi d)}{\pi} \right)^{-(k-1)} \sum_{t=0}^{n-1} S_{2,k}(n,s,t) \\
& \le
n^{-2d} 
\int_0^1 \frac{du}{u^d} \int_0^{\infty} 
\left\{ 
\int_u^1 \frac{dt }{(1 - t +  u + m_1)^2} 
\right\}
\left\{ \int_0^{\infty} D_k(m_1,\ell) f_{n,s}(\ell) d\ell \right\} dm_1 \\
&= 
n^{-2d} 
\int_0^1 \frac{du}{u^d} \int_0^{\infty} 
\left\{ 
\frac{1}{u + m_1} - \frac{1}{1 + m_1}
\right\} T^{k-1}f_{n,s} (m_1) dm_1 \\
&= n^{-2d} 
\left\{
(\psi, T^{k}f_{n,s})_L
- \frac{1}{1-d}(\zeta_{}, T^{k-1} f_{n,s})_L,
\right\},
\end{aligned}
\label{eq:S2ktt}
\]
where $\psi\in L$ is defined by $\psi(x) := x^{-d}1_{(0,1)}(x)$. 
Notice that $\Vert \psi \Vert_L = (1-2d)^{-1/2}$.

Combining, 
\[
\sum_{t=0}^{n-1} \{ S_{1,k}(n,s,t) + S_{2,k}(n,s,t) \}
\le n^{-2d} K_3 
\left (\frac{r\sin(\pi d)}{\pi} \right)^{k-1}
(\psi, T^{k}f_{n,s})_L.
\]
Since (\ref{eq:opnormT956}) and (\ref{eq:normf123}) imply
\[
(\psi, T^{k}f_{n,s})_L \leq \|\psi\|_L \|T\|^k \|f_{n,s}\|_L  
= \frac{\pi^k}{(1-4d^2)^{1/2}} \left(\frac{n}{s+1}\right)^{(1/2)+d},
\]
(\ref{eq:sumT}) follows.
\end{proof}

Now we are ready to prove Theorem \ref{thm:Bound1}.

\begin{proof}[Proof of Theorem \ref{thm:Bound1}]
Let $K_1$ and $K_2$ be as in Proposition \ref{prop:aA-est123} and 
$N_1\in \N$ be as in Lemma \ref{lemma:Sestim321a}.
Take $r\in (1,\infty)$ such that $r \sin(\pi d)<1$
Then, we define $M\in (0,\infty)$ by
\begin{equation*}
M := \frac{\pi K_1 K_2 K_3}{(1-4d^2)^{1/2} \{1-r\sin(\pi d)\}}.
\end{equation*}
Recall $R_{s,n}$ from (\ref{eq:RCn-b}). 
By Theorem \ref{thm:Tnbound}, $R_{s,n} \le A_{s,n} + B_{s,n}$ 
holds for $n\in \N$ and $s \in \{1, \dots,n\}$, 
where
\begin{align*}
A_{s,n} &:= K_1 K_2 \sum_{k=1}^{\infty} \sum_{t=1}^{n} \{ S_{1,2k-1}(n,n-s,t-1) + S_{2,2k-1}(n,n-s,t-1)\},\\
B_{s,n} &:= K_1 K_2 \sum_{k=1}^{\infty} \sum_{t=1}^{n} \{ S_{1,2k}(n,s-1,t-1) + S_{2,2k}(n,s-1,t-1)\}.
\end{align*}
By Lemma \ref{lemma:Sestim321a}, we have, 
for $n \in \{N_1, N_1+1, \dots\}$ and $s \in \{1,\dots,n\}$,
\[
\begin{aligned}
A_{s,n} 
&= 
K_1 K_2 \sum_{k=1}^{\infty} \sum_{t=0}^{n-1} \{S_{1,2k-1}(n,n-s,t) + S_{2,2k-1}(n,n-s,t)\} \\
&\leq
\frac{\pi K_1 K_2 K_3 }{(1-4d^2)^{1/2}}
\left[  \sum_{k=1}^{\infty}
 \{ r\sin(\pi d)\}^{k-1}  \right] \frac{n^{(1/2)-d}}{(n+1-s)^{(1/2)+d}} 
=
M  \frac{n^{(1/2)-d}}{(n+1-s)^{(1/2)+d}},
\end{aligned}
\]
and, similarly,
\[
B_{s,n} 
= K_1 K_2 \sum_{k=1}^{\infty} \sum_{t=0}^{n-1} \{S_{1,k}(n,s-1,t) + S_{2,k}(n,s-1,t)\}
\leq 
M  n^{(1/2)-d} s^{-(1/2)-d}.
\]
Combining, 
$R_{s,n} \leq
M  n^{(1/2)-d} \{ (n+1-s)^{-(1/2)-d} + s^{-(1/2)-d} \}$ 
for $n \in \{N_1, N_1+1, \dots\}$ and $s \in \{1, \dots,n\}$. 
Since $\{(n,s): 1\leq s \leq n \leq N_1\}$ is a finite set, 
(\ref{eq:Tnwtt}) holds for some $M_1 \in (0,\infty)$.
\end{proof}


\section{Proof of Theorem \ref{thm:Bound3}}\label{sec:7}

For $d\in (0,1/2)$, we put
\begin{equation}
\phi(x) := \int_{0}^{1} \frac{1}{(x +s)^{1+d}} ds, \quad x>0.
\label{eq:phi111}
\end{equation}
Recall the Hilbert space $L$ from (\ref{eq:L2def234}). 

\begin{proposition}\label{prop:phiL222}
For $d\in (0,1/2)$, we have $\phi \in L$.
\end{proposition}

\begin{proof}
Since 
$\phi(x) \le \int_0^1 x^{-1-d} ds 
= x^{-(1+d)}$ and 
$\phi(x) 
= (1/d) \{  x^{-d} - (x + 1)^{-d} \}$ for $x>0$, 
the assertion follows.
\end{proof}

Recall $S_{1,k}(n,s,t)$ and $S_{2,k}(n,s,t)$ from (\ref{eq:S1k}) and (\ref{eq:S2k}), 
respectively. 

\begin{lemma}\label{lemma:Sestim321b}
We assume (F${}_d$) for $d\in (0,1/2)$. 
Take $r\in (1,\infty)$ and let $K_3\in (0,\infty)$ and $N_1\in \N$ be as in 
Proposition \ref{prop:beta123}. Then, for $k\in\N$, $\ n\in \{N_1, N_1+1,\dots\}$ and 
$t \in \{0, \dots, n-1\}$, we have
\begin{equation}
\sum_{s=0}^{n-1} \{ S_{1,k}(n,s,t) + S_{2,k}(n,s,t) \} \\
\leq 
\frac{2 K_3 \|\phi\|_L \{ r\sin(\pi d)\}^{k-1} n^{1-d}}{(1-2d)^{1/2} (t+1)^{d} (n-t)}.
\label{eq:Sest768}
\end{equation}
\end{lemma}

\begin{proof}
The proof is similar to that of Lemma \ref{lemma:Sestim321a}. 
We prove (\ref{eq:Sest768}) only for $k\ge 2$; the case $k=1$ can be treated in 
the same way.

Let $n \ge N_1$ and $t \in \{0,\dots, n-1\}$. 
Using Proposition \ref{prop:btilineq111} and (\ref{eq:Dintegral111}), 
we have
\[
\begin{aligned}
&(t+1)^{d} \sum_{s=0}^{n-1} S_{1,k}(n,s,t)
=  
\int_0^{n} ds \int_0^{\infty} d\ell 
\frac{\Vert \tilde{b}^k_{n,1,[\ell]} \Vert }{([s] + [\ell] + 2)^{1+d}} \\
&\quad=
n^2 \int_0^{1} ds \int_0^{\infty} d\ell 
\frac{\Vert \tilde{b}^k_{n,1,[n \ell]} \Vert }{( [n s] +[n \ell] + 2)^{1+d}} 
\le 
n^{-d} \int_0^{\infty} \phi(\ell) (n \Vert \tilde{b}^k_{n,1,[n \ell]} \Vert)  d\ell \\
&\quad\le
n^{-d} K_3  \left( \frac{r\sin(\pi d)}{\pi} \right)^{k-1}
\int_0^{\infty} \frac{1}{1+m_1} \left\{ \int_0^{\infty} D_k(m_1,\ell) \phi(\ell) d\ell \right\} dm_1 \\
&\quad= n^{-d} K_3 \left( \frac{r\sin(\pi d)}{\pi} \right)^{k-1} 
(\zeta_{}, T^{k-1} \phi)_L,
\end{aligned}
\]
where $\zeta$ is as in the proof of Lemma \ref{lemma:Sestim321a}. 
Using (\ref{eq:opnormT956}) and 
$\Vert \zeta\Vert_L=1$. 
Therefore, we obtain
\begin{equation}
\begin{aligned}
\sum_{s=0}^{n-1} S_{1,k}(n,s,t) 
&\leq n^{-d} (t+1)^{-d} K_3 \left( \frac{r\sin(\pi d)}{\pi} \right)^{k-1} \|\zeta_{}\|_L \|T\|^{k-1} \|\phi_{}\|_L \\
&= n^{-d} (t+1)^{-d} \|\phi_{}\|_L K_3 \{r\sin(\pi d)\}^{k-1}.
\end{aligned}
\label{eq:S1kss}
\end{equation}

We turn to the sum of $S_{2,k}(n,s,t)$. Since $S_{2,k}(n,s,0)=0$, 
let $n\in \{N_1, N_1+1, \dots\}$ and $t \in \{1,\dots, n-1\}$. 
Then, 
\[
\begin{aligned}
& \sum_{s=0}^{n-1} S_{2,k}(n,s,t) 
= \sum_{s=0}^{n-1} \sum_{u=0}^{t-1} \sum_{\ell = 0}^{\infty} 
\frac{1}{(s + \ell +2)^{1+d}(u+1)^d} 
\Vert \tilde{b}^k_{n,t-u,\ell} - \tilde{b}^k_{n,t + 1 - u,\ell} \Vert\\
&=
\int_0^{n} ds \int_0^{t} du \int_0^{\infty} d\ell 
\frac{1}{( [s] + [\ell]+2)^{1+d}([u] + 1)^d} 
\Vert \tilde{b}^k_{n,t-[u],[\ell]} - \tilde{b}^k_{n,t + 1 - [u],[\ell]} \Vert\\
&\leq
n^2 \int_0^{1} ds \int_0^{t} du \int_0^{\infty} d\ell 
\frac{1}{( [ns] + [n \ell]+ 2)^{1+d} u^d}
\Vert \tilde{b}^k_{n,t-[u],[n \ell]} - \tilde{b}^k_{n,t + 1 - [u],[n \ell]} \Vert\\
&\le 
n^{-(1+d)}  \int_0^{t} \frac{du}{u^d} \int_0^{\infty} 
\phi_{}(\ell) (n^2 \Vert \tilde{b}^k_{n,t-[u],[n \ell]} - \tilde{b}^k_{n,t + 1 - [u],[n \ell]} \Vert)  d\ell,
\end{aligned}
\]
where the inequality 
$( [ns] + [n \ell]+ 2)^{-(1+d)} \leq n^{-(1+d)} (s+\ell)^{-(1+d)}$ 
is used in the last inequality. 
Therefore, using Proposition \ref{prop:btilineq222} and (\ref{eq:Dintegral111}), 
we obtain
\[
\begin{aligned}
& (K_3)^{-1}n^{1+d} \left( \frac{r\sin(\pi d)}{\pi} \right)^{-(k-1)} \sum_{s=0}^{n-1} S_{2,k}(n,s,t) \\
&\le 
\int_0^{t} \frac{du}{u^d} \int_0^{\infty} d\ell 
\phi_{}(\ell)\int_0^{\infty} \frac{D_k(m_1,\ell)}{\{1 - (t/n) + (u/n) + m_1\}^2}  dm_1 \\
&=
\int_0^{t} \frac{du}{u^d} 
\int_0^{\infty}
\frac{T^{k-1} \phi_{} (m_1)}{\{1 - (t/n) + (u/n) + m_1\}^2}  dm_1 
=
\int_0^{t} \frac{1}{u^d} ( g_{t,u}, T^{k-1}\phi_{} )_L du,
\end{aligned}
\]
where 
$g_{t,u}(m_1) := \{1 - (t/n) + (u/n) + m_1\}^{-2}$. 
Using
\begin{equation*}
\|g_{t,u}\|_L^2 = \int_0^{\infty} \frac{1}{\{1 - (t/n) + (u/n) + m_1\}^4} dm_1
= \frac{1}{3\{1-(t/n)+(u/n) \}^3}
\end{equation*}
as well as (\ref{eq:opnormT956}) and $3^{-1}<1$, we have
\[
\begin{aligned}
&\sum_{s=0}^{n-1} S_{2,k}(n,s,t) 
\leq
n^{-(1+d)} K_3  \left( \frac{r\sin(\pi d)}{\pi} \right)^{k-1} \|T\|^{k-1} \|\phi_{}\|_L
\int_0^{t} \frac{1}{u^d} \|g_{t,u}\|_L du\\
&\leq
n^{-(1+d)} K_3 \|\phi_{}\|_L \{ r\sin(\pi d)\}^{k-1}
\int_0^{t} \frac{1}{u^d} \frac{1}{\{1-(t/n)+(u/n) \}^{3/2}} du.
\end{aligned}
\]
The Cauchy--Schwartz inequality yields
\[
\begin{aligned}
&\int_0^{t} \frac{1}{u^d} \frac{1}{\{1-(t/n)+(u/n) \}^{3/2}} du 
\leq
\left(  \int_0^{t} \frac{du}{u^{2d}} \right)^{1/2} 
\left[ \int_0^{t} \frac{du}{\{1-(t/n)+(u/n) \}^3} \right]^{1/2} \\
&= (1-2d) ^{-1/2} t^{(1/2) - d} \left[ \frac{t\{2-(t/n)\}}{2\{1-(t/n)\}^2} \right]^{1/2}
\leq  \frac{(1+t)^{1-d}}{(1-2d)^{1/2} \{1-(t/n)\}},
\end{aligned}
\]
where $t<t+1$ and $0<2-(t/n)<2$ are used in the last inequality. 
Therefore,
\begin{equation}
\sum_{s=0}^{n-1} S_{2,k}(n,s,t) 
\leq \frac{(t+1)^{1-d} K_3 \|\phi\|_L \{ r\sin(\pi d)\}^{k-1}}{n^d (n-t) (1-2d)^{1/2}}.
\label{eq:S2kss3}
\end{equation}

Combining (\ref{eq:S1kss}) and (\ref{eq:S2kss3}) and using $(1-2d)^{1/2} < 1$ and 
$n+1 \le 2n$, we obtain
\[
\begin{aligned}
&\sum_{s=0}^{n-1} \{S_{1,k}(n,s,t) + S_{2,k}(n,s,t) \} 
\leq \frac{K_3 \|\phi\|_L}{(1-2d)^{1/2}} \{ r\sin(\pi d)\}^{k-1} 
n^{-d}(t+1)^{-d} \left\{ 1 + (1 + t) (n-t)^{-1}\right\} \\
&= \frac{K_3 \|\phi\|_L \{ r\sin(\pi d)\}^{k-1}}{(1-2d)^{1/2} (t+1)^{d}}  n^{-d}  
\left( \frac{n+1}{n-t} \right) 
\le \frac{2 K_3 \|\phi\|_L \{ r\sin(\pi d)\}^{k-1} n^{1-d}}{(1-2d)^{1/2} (t+1)^{d}
 (n-t)}
\end{aligned}
\]
for $n\in \{N_1, N_1+1, \dots\}$ and $t \in \{0, \dots,n-1\}$. 
Thus, (\ref{eq:Sest768}) follows.
\end{proof}

Now we are ready to prove Theorem \ref{thm:Bound3}.

\begin{proof}[Proof of Theorem \ref{thm:Bound3}]
Let $K_1$ and $K_2$ be as in Proposition \ref{prop:aA-est123} 
and $K_3$ as in Proposition \ref{prop:beta123}. 
Take $r\in (1,\infty)$ such that $r \sin(\pi d)<1$ and choose $N_1\in \N$ so that 
(\ref{eq:betaestim327}) holds. We define $M\in (0,\infty)$ by
\begin{equation*}
M := \frac{2K_1 K_2 K_3 \Vert \phi\Vert_L }{(1-2d)^{1/2} \{1-r\sin(\pi d)\}}.
\end{equation*}
By Theorem \ref{thm:Tnbound} and Lemma \ref{lemma:Sestim321b}, 
we have, for $n\in \N$ and $t \in \{1, \dots,n\}$,
\[
\begin{aligned}
&C_{t,n}
\leq K_1 K_2 
\sum_{k=1}^{\infty} \sum_{s=1}^{n} \{ S_{1,2k-1}(n,n-s,t-1) + S_{2,2k-1}(n,n-s,t-1)\}\\
&\quad\quad+
K_1 K_2 
\sum_{k=1}^{\infty} \sum_{s=1}^{n} \{ S_{1,2k}(n,s-1,t-1) + S_{2,2k}(n,s-1,t-1)\}. 
\\
&=  K_1 K_2 \sum_{k=1}^{\infty} 
\sum_{s=0}^{n-1} \{ S_{1,k}(n,s,t-1) + S_{2,k}(n,s,t-1)\}\\
&\leq  \frac{2 K_1 K_2 K_3 \|\phi\|_L}{(1-2d)^{1/2}}
\left(  \sum_{k=1}^{\infty}
 \{ r\sin(\pi d)\}^{k-1}  \right)
\frac{n^{1-d}}{ t^{d} (n+1-t)} 
=  \frac{M n^{1-d}}{ t^{d} (n+1-t)}.
\end{aligned}
\]
Since $\{(n,t): 1\leq t \leq n \leq N_1\}$ is a finite set, 
(\ref{eq:Tnwss}) holds for some $M_3 \in (0,\infty$).
\end{proof}


\section{Proofs of Theorems \ref{thm:Bound5} and \ref{thm:Bound6}} \label{sec:8}

We continue to assume that $w$ satisfies (F$_d$) with $d\in (0,1/2)$. 
Let $\tilde{w}$ be as in (\ref{eq:wtilde}). 
Recall the sequences $\{a_k\}, \{\tilde{a}_k\} \in \ell_1^{q\times q}$ from Section \ref{sec:4}. 
Then, we have
\begin{align}
(T_{\infty}(w)^{-1})^{s,t} 
&= \sum_{\ell=1}^{s\wedge t} \tilde{a}_{s-\ell}^* \tilde{a}_{t-\ell}, \quad 
s,t\in\N,  \label{eq:Tinftyinv973}  \\
(T_{\infty}(\tilde{w})^{-1})^{s,t} 
&= \sum_{\ell=1}^{s\wedge t} a_{s-\ell}^* a_{t-\ell}, \quad 
s,t\in\N,  \label{eq:Tinftyinvtilde}
\end{align}
where $s\wedge t := \min(s,t)$. 
See, e.g., \cite{I23}, Remark 2.1.

\begin{proposition}\label{prop:Tinfinv123}
We assume (F${}_d$) for $d\in (0,1/2)$. 
Then there exists $B_1 = B_1(w) \in (0,\infty)$, which depends only on $w$, 
such that
\begin{equation}
\sum_{s=1}^{\infty} \Vert (T_{\infty}(w)^{-1})^{s,t}\|
\le B_1,\quad t\in\N.
\label{eq:Tinfty1norm}
\end{equation}
\end{proposition}

\begin{proof}
By (\ref{eq:Tinftyinv973}), $\sum_{s=1}^{\infty} \Vert (T_{\infty}(w)^{-1})^{s,t}\|$ 
is bounded by
\[
\begin{aligned}
&\sum_{s=1}^{\infty} \sum_{\ell=1}^{s\wedge t} 
\|\tilde{a}_{s-\ell}\| \|\tilde{a}_{t-\ell}\|
= \sum_{s=1}^{t} \sum_{\ell=1}^{s} 
\|\tilde{a}_{s-\ell}\| \|\tilde{a}_{t-\ell}\|
+
\sum_{s=t+1}^{\infty} \sum_{\ell=1}^{t} 
\|\tilde{a}_{s-\ell}\| \|\tilde{a}_{t-\ell}\|\\
&
= \sum_{\ell=1}^{t} \|\tilde{a}_{t-\ell}\| 
\sum_{s=\ell}^{t} \|\tilde{a}_{s-\ell}\| 
+
\sum_{\ell=1}^{t} \|\tilde{a}_{t-\ell}\|
\sum_{s=t+1}^{\infty} \|\tilde{a}_{s-\ell}\|
= \left(\sum_{\ell=0}^{t-1} \|\tilde{a}_{\ell}\| \right) 
\left(\sum_{s=0}^{\infty} \|\tilde{a}_s\| \right)
\end{aligned}
\]
for $t\in\N$, 
hence (\ref{eq:Tinfty1norm}) holds with $B_1 = ( \sum_{k=0}^{\infty} \|\tilde{a}_k\| )^2$.
\end{proof}

\begin{proposition}\label{prop:Tbound22}
We assume (F${}_d$) for $d\in (0,1/2)$. Let $0<\delta < r<1$. 
\begin{itemize}
\item[(i)] There exists $C_1 \in (0,\infty)$, which depends only on $w$, $\delta$ and $r$, such that
\begin{equation}  \label{eq:deltabound2}
\sum_{s=1}^{[n\delta]}\| (T_{\infty}(w)^{-1})^{s,t}\| \leq C_1 n^{-d}, \quad n \in \N, ~~t \in \{[r n]+1, \dots,n\}.
\end{equation}

\item[(ii)] There exists $C_2 \in (0,\infty)$, which depends only on $\tilde{w}$, $\delta$ and $r$, such that
\begin{equation}  \label{eq:deltabound222}
\sum_{s=1}^{[n\delta]}\| (T_{\infty}(\tilde{w})^{-1})^{n+1-s,n+1-t}\| \leq C_2 n^{-d}, \quad n \in \N, ~~t \in \{[r n]+1, \dots,n\}.
\end{equation}
\end{itemize}

\end{proposition}

\begin{proof}
Let $s\in \{1, \dots, [n\delta]\}$ and $t \in \{[rn]+1, \dots, n\}$. Then, we have
\begin{equation}
t-s \ge t-[n\delta] \ge [rn]+1-[n\delta] \geq n(r-\delta) >0.
\label{eq:tsdifference}
\end{equation}

We first show (i). By (\ref{eq:Tinftyinv973}), (\ref{eq:tsdifference}) and Proposition \ref{prop:aA-est123}(i),
\[
\begin{aligned}
&\sum_{s=1}^{[n\delta]}\| (T_{\infty}(w)^{-1})^{s,t}\| 
\leq \sum_{s=1}^{[n\delta]} \sum_{\ell =1}^{s} \| \tilde{a}_{s-\ell}\| \| \tilde{a}_{t-\ell}\| 
= \sum_{\ell=1}^{[n\delta]}  \| \tilde{a}_{t-\ell}\| \sum_{s=\ell}^{[n\delta]} \| \tilde{a}_{s-\ell}\| \\
&\quad\quad\quad 
\leq \left\{  \sum_{k=0}^{\infty} \| \tilde{a}_{k}\| \right\} \sum_{u=t-[n\delta]}^{\infty}  \| \tilde{a}_{u}\|.\leq \tilde{C}_1 (t-[n\delta])^{-d}
\le \tilde{C}_1 (r-\delta)^{-d} n^{-d},
\end{aligned}
\]
where $\tilde{C}_1$ is a positive constant that depends only on $w$. 
Therefore, (\ref{eq:deltabound2}) holds with $C_1 = \tilde{C}_1 (r-\delta)^{-d}$.

Next, we show (ii). Since $n+1-s > n+1-t$ and
\[
n+1-s-\ell \ge n+1-[n\delta]- (n+1-t) = t- [n\delta] \ge n(r-\delta)
\]
for $\ell \in \{1, \dots, n+1-t\}$, we see from (\ref{eq:Tinftyinvtilde}) and 
Proposition \ref{prop:aA-est123}(i) that
\[
\begin{aligned}
&\sum_{s=1}^{[n\delta]} \| (T_\infty(\tilde{w})^{-1})^{n+1-s,n+1-t}\| 
\leq
 \sum_{\ell = 1}^{n+1-t} \|a_{n+1-t-\ell}\|  
 \sum_{s=1}^{[n\delta]}  \|a_{n+1-s-\ell}\|\\
&\quad\quad\quad \leq 
\left\{\sum_{k = 0}^{\infty} \|a_k\|\right\}  \sum_{u=t-[n\delta]}^{\infty}  \|a_{u}\| 
\leq \tilde{C}_2 (t-[n \delta])^{-d}
\leq \tilde{C}_2 (r-\delta)^{-d}n^{-d},
\end{aligned}
\]
where $\tilde{C}_2$ is a positive constant that depends only on $w$. 
Thus, (\ref{eq:deltabound222}) holds with $C_2 = \tilde{C}_2 (r-\delta)^{-d}$.
\end{proof}

Now, recall $\tilde{w}$ in (\ref{eq:wtilde}). Below shows the relationship between $T_n(w)^{-1}$ and $T_n(\tilde{w})^{-1}$.

\begin{proposition} \label{prop:time-reverse}
For $n \in \N$ and $s,t \in \{1, \dots, n\}$, we have 
$(T_n(w)^{-1})^{n+1-s,n+1-t} = (T_n(\tilde{w})^{-1})^{s,t}$.
\end{proposition}

\begin{proof}
Let $\tilde{\gamma}$ be the autocovariance function of $\{\tilde{X}_k\}$: 
$\tilde{\gamma}(k) 
= (2\pi)^{-1} \int_{-\pi}^{\pi} e^{-ik\theta} \tilde{w}(e^{i\theta})d\theta$. 
Since $\tilde{\gamma}(k)=\gamma(-k)$ for $k\in\Z$, 
we have, for $n \in \N$ and $s, t \in \{1,\dots,n\}$,
\[
\begin{aligned}
&\sum_{k=1}^n T_n(\tilde{w})^{s, k} (T_n(w)^{-1})^{n+1-k, n+1-t}
=\sum_{k=1}^n \tilde{\gamma}(s-k)(T_n(w)^{-1})^{n+1-k, n+1-t}\\
&=\sum_{k=1}^n \gamma(k-s)(T_n(w)^{-1})^{n+1-k, n+1-t}
=\sum_{\ell=1}^n \gamma(n+1-s-\ell)(T_n(w)^{-1})^{\ell, n+1-t}\\
&=\sum_{\ell=1}^n T_n(w)^{n+1-s,\ell}(T_n(w)^{-1})^{\ell, n+1-t}
=\delta_{n+1-s, n+1-t}I_q = \delta_{s, t}I_q.
\end{aligned}
\]
Thus, we get the desired result.
\end{proof}

Now we are ready to prove Theorems \ref{thm:Bound5} and \ref{thm:Bound6}.

\begin{proof}[Proof of Theorem \ref{thm:Bound5}]
Let $B_1$ be as in Proposition \ref{prop:Tinfinv123}. 
In this proof, we write $B_2$ and $B_3$ for positive constants which depend only on $w$. 

Since 
$t^{-d} \leq 1$ and 
$(n+1-t)^{-1} \leq (n/2)^{-1}$ for $t \in \{1, \dots, [(n+1)/2]\}$, 
it follows from Theorem \ref{thm:Bound3} that
\begin{equation}
\sum_{s=1}^{n} \| (T_n(w)^{-1})^{s,t} - (T_{\infty}(w)^{-1})^{s,t} \| \le B_2,
\quad n\in\N,\ t\in \{1, \dots, [(n+1)/2]\}.
\label{eq:diffest717}
\end{equation}
The estimates (\ref{eq:Tinfty1norm}) and (\ref{eq:diffest717}), 
together with triangle inequality, yield
\begin{equation}
\sum_{s=1}^{n} \Vert (T_n(w)^{-1})^{s,t}\| \leq B_1+B_2, 
\quad n\in \N, \ \ t \in \{1, \dots,[(n+1)/2]\}.
\label{eq:Tnwsum11}
\end{equation}

To estimate $\sum_{s=1}^{n} \Vert (T_n(w)^{-1})^{s,t}\|$ 
for $t > (n+1)/2$, we use the properties of the time-reversed process. 
Recall $\tilde{w}$ from (\ref{eq:wtilde}). 
Then, in the same way that (\ref{eq:Tnwsum11}) was obtained, we obtain
\begin{equation}
\sum_{s=1}^{n} \Vert (T_n(\tilde{w})^{-1})^{s,t}\| \leq B_3, 
\quad n\in \N, \ \ t \in \{1, \dots,[(n+1)/2]\}.
\label{eq:Tnwsum22}
\end{equation}
Next, from Proposition \ref{prop:time-reverse}, for $n \in N$ and $t \in \{1, \dots, n\}$,
\begin{equation*}
\sum_{s=1}^{n} \Vert (T_n(w)^{-1})^{s,n+1-t}\| 
= \sum_{s=1}^{n} \Vert (T_n(w)^{-1})^{n+1-s,n+1-t}\| 
= \sum_{s=1}^{n} \Vert (T_n(\tilde{w})^{-1})^{s,t}\|.
\end{equation*}
From this and (\ref{eq:Tnwsum22}), we see that
\begin{equation}
\sum_{s=1}^{n} \Vert (T_n(w)^{-1})^{s,n+1-t}\| \leq B_3, 
\quad n\in \N, \ \ t \in \{1, \dots,[(n+1)/2]\}.
\label{eq:Tnwsum33}
\end{equation}

Combining (\ref{eq:Tnwsum11}) and (\ref{eq:Tnwsum33}), we have 
$\sum_{s=1}^{n} \Vert (T_n(w)^{-1})^{s,t}\| \leq B_4$ for 
$n\in \N$ and $t\in \{1,\dots,n\}$, 
where $B_4:=\max(B_1+B_2, B_3)$. 
This and (\ref{eq:Tinfty1norm}) yield 
\begin{equation*}
C_{t,n} \leq
\sum_{s=1}^{n} \| (T_n(w)^{-1})^{s,t}\| +\sum_{s=1}^{n}\| (T_{\infty}(w)^{-1})^{s,t}\| 
\le B_1 + B_4
\end{equation*}
for $n\in \N$ and $ t\in \{1, \dots, n\}$. 
Thus (\ref{eq:col-sum3}) holds with $M_5 = B_1 + B_4$.
\end{proof}

\begin{proof}[Proof of Theorem \ref{thm:Bound6}]
Let $r := (\delta+1)/2$ so that $0< \delta < r < 1$. 
First, by Theorem \ref{thm:Bound4}, there exists $M_4(w,r) \in (0,\infty)$ such that
\begin{equation} \label{eq:deltabound1}
\sum_{s=1}^{[n\delta]} \Delta_n^{s,t} \leq C_{t,n} \leq M_4(w,r) n^{-d}, \qquad n \in \N,~~ t \in \{1, \dots, [r n]\}.
\end{equation}

Next, we assume $t \in \{[r n]+1, \dots,n\}$. Then, by using (\ref{eq:TnTtilde}), we have
\[
\begin{aligned}
\sum_{s=1}^{[n\delta]} \Delta_n^{s,t} 
&\leq 
\sum_{s=1}^{[n\delta]} \| (T_n(w)^{-1})^{s,t}\| +\sum_{s=1}^{[n\delta]}\| (T_{\infty}(w)^{-1})^{s,t}\| \\
&= \sum_{s=1}^{[n\delta]} \| (T_n(\tilde{w})^{-1})^{n+1-s,n+1-t}\| +\sum_{s=1}^{[n\delta]}\| (T_{\infty}(w)^{-1})^{s,t}\| \\
&\leq \sum_{s=1}^{[n\delta]} \| (T_n(\tilde{w})^{-1})^{n+1-s,n+1-t} - 
(T_\infty(\tilde{w})^{-1})^{n+1-s,n+1-t}\| \\
&\qquad +  \sum_{s=1}^{[n\delta]} \|(T_\infty(\tilde{w})^{-1})^{n+1-s,n+1-t}\|
+\sum_{s=1}^{[n\delta]}\| (T_{\infty}(w)^{-1})^{s,t}\| := I_1 +I_2 +I_3.
\end{aligned}
\]
First, we bound $I_1$. 
Since $\tilde{w}$ also satisfies (F$_d$) for the same $d\in (0,1/2)$ as $w$ 
and we have $n+1-t \leq n-[r n] \leq n(1-r)+1 < [n (1-\delta)]$ for large enough $n \in \N$, 
by applying Theorem \ref{thm:Bound4} to the pair $(\tilde{w},1 - \delta)$, we see that
\begin{equation} \label{eq:I1}
I_1
\leq M_4 (\tilde{w}, 1-\delta) n^{-d}, \qquad n \in \N,~~ t \in \{[r n]+1, \dots,n\},
\end{equation}
holds for some $M_4 (\tilde{w}, 1-\delta) \in (0,\infty)$. 
On the other hand, by Proposition \ref{prop:Tbound22}, $I_2$ and $I_3$ are bouned with
\begin{equation} \label{eq:I23}
I_2 \leq C_2 n^{-d} \quad \text{and} \qquad I_3 \leq C_1 n^{-d}
, \qquad n \in \N,~~ t \in \{[r n]+1, \dots,n\},
\end{equation}
respectively. 
Combining (\ref{eq:I1}) and (\ref{eq:I23}), we obtain
\begin{equation} 
\sum_{s=1}^{[n\delta]} \Delta_n^{s,t} 
\leq \{ M_4 (\tilde{w}, 1-\delta) + C_1 + C_2\}
n^{-d}, \qquad n \in \N,~~ t \in \{[r n]+1, \dots,n\}.
\label{eq:deltabound3}
\end{equation}

If we put $M_6 := \max( M_4(w,r) , M_4 (\tilde{w}, 1-\delta) + C_1 + C_2)$, 
then, by (\ref{eq:deltabound1}) and (\ref{eq:deltabound3}), 
(\ref{eq:col-sum4}) holds. 
\end{proof}


\section{Proofs of Theorems \ref{thm:B-type1} and \ref{thm:B-type2}}\label{sec:9}

For $\mathbf{y} = (y_1^\top ,y_2^\top, \dots)^\top \in \C^{\infty\times q}$ and 
$n\in\N$, we define
\[
I_n := \sum_{s=1}^{n}  \sum_{t=1}^n \Delta_n^{s,t} \|y_t\|,
\quad \mbox{and} \quad 
J_n := \sum_{s=1}^{n}  \sum_{t=n+1}^{\infty} \| (T_\infty(w)^{-1})^{s,t}\| \cdot \| y_t\|,
\]
where $\Delta_n^{s,t}$ is as in (\ref{eq:Delta}). 
By (\ref{eq:Tinvsa333}) and Proposition \ref{prop:Tinfinv123}, 
$J_n<\infty$ holds if 
$\sup_{t\in\N} \Vert y_t\Vert<\infty$, hence, in particular, if 
$\mathbf{y} \in \mathcal{A}_{\rho}^{\infty \times q}$ for $\rho \in (1-2d, \infty)$ or $\mathbf{y} \in \mathcal{B}^{\infty \times q}$.

For $d\in (0,1/2)$, we put
\begin{equation}
\xi(d) := \int_0^1 ds \int_1^{\infty} dt \int_0^s du 
\frac{1}{u^d (t-s+u)^{2+d}},
\label{eq:xi257}
\end{equation}
which can be calculated as follows:
\[
\begin{aligned}
\xi(d) 
&= \int_0^1 ds \int_1^{\infty} dt \int_0^s 
\frac{du}{u^d (t-s+u)^{2+d}}
=\frac{1}{1+d} \int_0^1 ds \int_0^s 
\frac{du}{u^d (1-s+u)^{1+d}}\\
&=\frac{1}{1+d} \int_0^1 \frac{du}{u^d} \int_u^1 
\frac{ds}{(1-s+u)^{1+d}}
= \frac{1}{(1-d^2)(1-2d)}.
\end{aligned}
\]
In particular, $\xi(d)<\infty$.

Recall the sequence $\{\tilde{a}_k\} \in \ell_1^{q\times q}$ from Section \ref{sec:4}. 
We also need the following proposition to prove Lemma \ref{lemma:I12nA} below.

\begin{proposition}\label{prop:atilde382}
We assume $(\mathrm{F}_{d})$ for $d\in (0,1/2)$. 
Then there exists $K_4\in (0,\infty)$ such that
\begin{equation}
\Vert \tilde{a}_{n+1} - \tilde{a}_n \Vert \le K_4 (n+2)^{-2-d}, 
\quad n\in\N.
\label{eq:atildediff634}
\end{equation}
\end{proposition}

\begin{proof}
Recall $g_{\sharp}$ from Section \ref{sec:4}. 
By \cite[Lemma 2.2]{I02}, 
\[
\left\Vert n^{2+d}(\tilde{a}_n - \tilde{a}_{n-1}) + \frac{1}{\Gamma(-1-d)}\{g_{\sharp}(1)^*\}^{-1} \right\Vert
= O(n^{-1}),
\quad n\to\infty
\]
In particular, 
$\lim_{n\to\infty} n^{2+d}\Vert \tilde{a}_n - \tilde{a}_{n-1} \Vert 
= \Vert \{g_{\sharp}(1)^*\}^{-1}\Vert / \Gamma(-1-d)$. 
This implies the desired assertion. 
\end{proof}

\begin{lemma}\label{lemma:I12nA}
We assume (F$_d$) for $d\in (0,1/2)$. 
Let $\mathbf{y} \in \mathcal{A}_{\rho}^{\infty \times q}$ for $\rho \in (1-2d, \infty)$. 
Then the following two assertions hold:
\begin{align}
I_n &=
\begin{cases}
O(n^{1-2d-\rho}) &(1-2d<\rho<1-d),\\
O(n^{-d}\log n) &(\rho=1-d),\\
O(n^{-d}) &(\rho>1-d),
\end{cases}
\quad n\to\infty,
\label{eq:Arhobound-1}
\\
J_n &= O(n^{1-2d-\rho}), \quad n\to\infty.
\label{eq:Arhobound-2}
\end{align} 
\end{lemma}

\begin{proof}
Since $\mathbf{y} \in \mathcal{A}_{\rho}^{\infty \times q}$ with $\rho \in (1-2d, \infty)$, 
there exists $R\in (0,\infty)$ such that
\begin{equation}
\|y_t\| \leq  R t^{-\rho}, \quad t \in \N.
\label{eq:yest242}
\end{equation}

We first prove (\ref{eq:Arhobound-1}). 
Recall $C_{t,n}$ from (\ref{eq:RCn-a}). Fix $\delta \in (0,1)$. 
Then, by (\ref{eq:yest242}) and Theorems \ref{thm:Bound3} and \ref{thm:Bound4}, 
we have $I_n = \sum_{t=1}^{n} C_{t,n} \|y_t\| = O(I_{1,n}) + O( I_{2,n})$ as $n\to\infty$, 
where 
$I_{1,n} := n^{-d} \sum_{t=1}^{[\delta n]} t^{-d-\rho}$ and 
$I_{2,n} := n^{(1/2)-d} \sum_{t=[\delta n]+1}^{n} (n+1-t)^{-(1/2)-d} t^{-\rho}$. 
It is straightforward to see that (\ref{eq:Arhobound-1}) with $I_n$ replaced by 
$I_{1,n}$ holds. 
On the other hand, since $1/t \leq 1/(\delta n)$ for 
$t \in \{[\delta n]+1, \dots,n\}$, $n^{-(1/2)+d} (\delta n)^{\rho} I_{2,n}$ 
is bounded by
\[
\sum_{t=[\delta n]+1}^{n}  (n+1-t)^{-(1/2)-d} \\
= \sum_{t=1}^{[\delta n]}  t^{-(1/2)-d} 
=O(n^{(1/2)-d}), 
\]
hence $I_{2,n} = O(n^{1-2d-\rho})$ as $n\to\infty$.
Combining, we obtain (\ref{eq:Arhobound-1}).

Next, we prove (\ref{eq:Arhobound-2}). Recall $\tilde{A}_k$ from (\ref{eq:tildeA656}). 
As in the proof of Proposition \ref{prop:sum-by-parts134}, 
by (\ref{eq:Tinftyinv973}), Proposition \ref{prop:aA-est123} (iii) and summation by parts, 
$(T_{\infty}(w)^{-1})^{s,t}$ is equal to 
\[
\sum_{\ell=1}^s \tilde{a}_{s-\ell}^* \tilde{a}_{t-\ell}
= \sum_{u=0}^{s-1} \tilde{a}_u^* \tilde{a}_{t-s+u}
= 
\tilde{A}_s^* \tilde{a}_{t-1}
- \sum_{u=0}^{s-2} \tilde{A}_{u+1}^* \left( \tilde{a}_{t-s+u+1} - \tilde{a}_{t-s+u}\right)
\]
for $s\le t$, 
hence 
$J_n \le J_{1,n} + J_{2,n}$ for $n\in\N$, 
where
\[
J_{1,n}:= \left( \sum_{s=1}^n \Vert \tilde{A}_s \Vert\right) 
\sum_{t=n+1}^{\infty} \Vert \tilde{a}_{t-1} \Vert \cdot \Vert y_t \Vert 
\ \ \mbox{and}\ \ 
J_{2,n}:= \sum_{s=1}^n \sum_{t=n+1}^{\infty} \sum_{u=0}^{s-2}
\Vert \tilde{A}_{u+1}\Vert \cdot \Vert \tilde{a}_{t-s+u+1} - \tilde{a}_{t-s+u} \Vert 
\cdot \Vert y_t \Vert.
\]

By Proposition \ref{prop:aA-est123} and (\ref{eq:yest242}), 
we have $J_{1,n}
=O( \sum_{s=1}^n s^{-d})O( \sum_{t=n+1}^{\infty} t^{-1-d-\rho})
=O(n^{1-2d-\rho})$ as $n\to\infty$. 
On the other hand, 
by (\ref{eq:yest242}) and Propositions \ref{prop:aA-est123} 
and \ref{prop:atilde382}, 
\[
\begin{aligned}
&\frac{J_{2,n}(n+1)^{\rho}}{K_2K_4R}
\le 
\sum_{s=0}^{n-1} \sum_{t=n}^{\infty} \sum_{u=0}^{s-1} \frac{1}{(u+1)^d(t-s+u+2)^{2+d}}\\
&= 
\int_0^n ds \int_n^{\infty} dt \int_0^{[s]} \frac{du}{([u]+1)^d([t]-[s]+[u]+2)^{2+d}}\\
&=n^3 
\int_0^1 ds \int_1^{\infty} dt \int_0^{[ns]/n} \frac{du}{([nu]+1)^d([nt]-[ns]+[nu]+2)^{2+d}}
\le n^{1-2d} \xi(d),
\end{aligned}
\]
hence 
$J_{2,n} = O(n^{1-2d-\rho})$ as $n\to\infty$. 
Combining, we obtain (\ref{eq:Arhobound-2}).
\end{proof}

\begin{lemma}\label{lemma:I12nB}
We assume (F$_d$) for $d\in (0,1/2)$. Let $\mathbf{y} \in \mathcal{B}^{\infty \times q}$. 
Then, for any $\kappa \in (0,d/(1-d))$, 
the following two assertions hold: as $n\to\infty$,
\begin{equation}
I_n = O (n^{-d+\kappa(1-d)} ) + O\left( \sum_{t=[n^{\kappa}]+1}^{n} \|y_t\|\right)
\quad \mbox{and}\quad 
J_n = O\left( \sum_{t=n+1}^{\infty} \|y_t\|\right), \quad n \to \infty.
\label{eq:Bbound-2}
\end{equation} 
\end{lemma}

\begin{proof}
In this proof, $C$ denotes a generic positive constant that can vary from line to line.

The second assertion in (\ref{eq:Bbound-2}) follows from Proposition \ref{prop:Tinfinv123}. 
Since $\kappa <1$, we have, for large enough $n\in\N$, $n^{\kappa} < n/2$, 
hence, by Theorem \ref{thm:Bound4} with $\delta = 1/2$, 
$C_{t,n} \leq C n^{-d} t^{-d}$ holds for $t \in \{1, \dots, [n^{\kappa}]\}$. 
Moreover, $\sup_{t \in \N} \|y_t\| < \infty$. Therefore, by Theorem \ref{thm:Bound5},
\[
\begin{aligned}
I_n &= \sum_{t=1}^{n} C_{t,n} \|y_t\|
\leq
C \sum_{t=1}^{[n^{\kappa}]} n^{-d} t ^{-d} \|y_t\| 
+ C \sum_{t=[n^{\kappa}]+1}^{n} \|y_t\|\\
&\leq C n^{-d} \sum_{t=1}^{[n^{\kappa}]} t ^{-d}
+ C \sum_{t=[n^{\kappa}]+1}^{n} \|y_t\|
=O (n^{-d+\kappa(1-d)} ) + O\left( \sum_{t=[n^{\kappa}]+1}^{n} \|y_t\|\right)
\end{aligned}
\]
as $n \to \infty$. Thus the first assertion in (\ref{eq:Bbound-2}) follows. 
\end{proof}

Now we are ready to prove Theorems \ref{thm:B-type1} and \ref{thm:B-type2}.

\begin{proof}[Proofs of Theorems \ref{thm:B-type1} and \ref{thm:B-type2}]
The solutions 
$\mathbf{z}_n = (z_{1,n}^\top, \dots, z_{n,n}^\top)^\top \in \C^{nq\times q}$ 
of (\ref{eq:tsysn123}) 
and 
$\mathbf{z} = (z_1^\top, z_2^\top, \dots)^\top \in \C^{\infty\times q} $
of (\ref{eq:tsysinf123}) 
are given by
\[
z_{s,n} =  \sum_{t=1}^{n} (T_{n}(w)^{-1})^{s,t} y_t, \quad s \in \{1, \dots, n\}
\quad \mbox{and}\quad 
z_s = \sum_{t=1}^{\infty} (T_\infty(w)^{-1})^{s,t} y_t, \quad s \in \N,
\]
respectively. We thus have
$
\|z_s- z_{s,n}\|
\le 
\sum_{t=1}^{n} \Delta_n^{s,t} \|y_t\| + 
\sum_{t=n+1}^{\infty} \| (T_\infty(w)^{-1})^{s,t}\| \cdot \| y_t\|
$
 for $s \in \{1,\dots,n\}$, 
hence 
$\sum_{s=1}^{n} \|z_s- z_{s,n}\| \leq I_n + J_n$. 
Therefore, Theorems \ref{thm:B-type1} and \ref{thm:B-type2} follow from 
Lemmas \ref{lemma:I12nA} and \ref{lemma:I12nB}, respectively.
\end{proof}

\section{Proof of Theorem \ref{thm:Omega}}\label{sec:10}
Since $T_n(w)^{-1}$  and $T_n(\tilde{w})^{-1}$ are Hermitian, so is $\Omega_{n,\delta}(w)$. This shows (i).

To prove (ii), we define the column sums
\begin{equation*}
D_{t,n} := \sum_{s=1}^{n} \Vert (T_n(w)^{-1})^{s,t} - (\Omega_{n,\delta}(w))^{s,t}\|, \qquad 
t \in \{1, \dots ,n\} \quad n\in \N.
\end{equation*} We will bound $D_{t,n}$ for each $t \in \{1, \dots, n\}$. 
Recall $H_{\delta}(n)$ and $T_{\delta}(n)$ from (\ref{eq:AdeltaBdelta}).

\vspace{1em}

\noindent {\bf Case 1:} $t \in H_{\delta}(n)$. \hspace{0.1em} In this case, we have
\[
\begin{aligned}
D_{t,n} &\leq \sum_{s \in H_{\delta}(n)} \Vert (T_n(w)^{-1})^{s,t} - (T_\infty(w)^{-1})^{s,t}\|
+ \frac{1}{2} \sum_{s \in H_{\delta}(n)^{\text{c}}} \Vert (T_n(w)^{-1})^{s,t} - (T_\infty(w)^{-1})^{s,t}\| \\
&\qquad\qquad 
+\frac{1}{2} \sum_{s \in H_{\delta}(n)^{\text{c}}} \Vert (T_n(w)^{-1})^{s,t} - (T_\infty(\tilde{w})^{-1})^{n+1-s,n+1-t}\|
=: E_{1,1} + E_{1,2} + E_{1,3}.
\end{aligned}
\]
We bound each term above. First, by Theorem \ref{thm:Bound4}, we have
\begin{equation} \label{eq:E1bound01}
E_{1,1} + E_{1,2} \leq \sum_{s=1}^{n} \Vert (T_n(w)^{-1})^{s,t} - (T_\infty(w)^{-1})^{s,t}\|
\leq M_4 n^{-d}, \quad n \in \N, ~~ t \in H_\delta (n).
\end{equation}
Next, we use (\ref{eq:TnTtilde}) to obtain
\[
\begin{aligned}
E_{1,3} &= \frac{1}{2} \sum_{s \in H_{\delta}(n)^{\text{c}}} \Vert (T_n(\tilde{w})^{-1})^{n+1-s,n+1-t} - (T_\infty(\tilde{w})^{-1})^{n+1-s,n+1-t}\| \\
&= \frac{1}{2} \sum_{s \in H_\delta(n)} \Vert (T_n(\tilde{w})^{-1})^{s,n+1-t} - (T_\infty(\tilde{w})^{-1})^{s,n+1-t}\|.
\end{aligned}
\]
From this and Theorem \ref{thm:Bound6} applied to $\tilde{w}$, we have
\begin{equation}\label{eq:E1bound02}
2E_{1,3} \leq \tilde{M}_6 n^{-d}, \qquad n \in \N, ~~ t \in \{1, \dots, n\}.
\end{equation}
Combining (\ref{eq:E1bound01}) and (\ref{eq:E1bound02}), we see that $D_{t,n}$ is bounded with
\begin{equation} \label{eq:Dtnbound1}
D_{t,n} \leq \{M_4 + (\tilde{M}_6/2)\} n^{-d}, \qquad n \in \N, ~~ t \in H_\delta (n).
\end{equation}

\vspace{1em}

\noindent {\bf Case 2:} $t \in T_{\delta}(n) = \{n-[\delta n]+1, \dots,n\}$. \hspace{0.1em} 
In a similar way as above, we can show
\begin{equation}\label{eq:Dtnbound2}
D_{t,n} \leq \{\tilde{M}_4 + (M_6/2)\} n^{-d}, \qquad n \in \N, ~~ t \in T_\delta (n).
\end{equation}

\vspace{1em}

\noindent {\bf Case 3:} $t \in I_{\delta}(n):= \{[\delta n]+1, \dots, n-[\delta n]\}$. \hspace{0.1em} 
By using (\ref{eq:TnTtilde}), we have
\[
\begin{aligned}
2D_{t,n} 
&\leq \sum_{s =1}^{n} \| (T_n(w)^{-1})^{s,t} - (T_\infty(w)^{-1})^{s,t}\|
+ \sum_{s =1}^{n} \| (T_n(w)^{-1})^{s,t} - (T_\infty (\tilde{w})^{-1})^{n+1-s,n+1-t}\| \\
&= C_{t,n} + \widetilde{C}_{n+1-t,n},
\end{aligned}
\]
where $C_{t,n}$ is defined as in (\ref{eq:RCn-a}) and $\widetilde{C}_{t,n}$ is defined 
similarly, replacing $w$ with $\tilde{w}$. 
We note that for large enough $n \in \N$, $t \in I_{\delta}(n)$ implies $t \in H_{\tilde{\delta}}(n)$ and $n+1-t \in H_{\tilde{\delta}}(n)$, where $\tilde{\delta} = 1-(\delta/2)$. 
Therefore, by using Theorem \ref{thm:Bound4} for both $(w,\tilde{\delta})$ and 
$(\tilde{w},\tilde{\delta})$, we have, for large enough $n\in \N$, 
\begin{equation} \label{eq:Dtnbound3}
D_{t,n} \leq (1/2)\{C_{t,n}+ \widetilde{C}_{n+1-t,n}\} \leq \{  (M_4/2) + (\tilde{M}_4/2)\}n^{-d}, \qquad t \in I_{\delta}(n).
\end{equation}

Since $\delta$ is fixed, combining (\ref{eq:Dtnbound1}), (\ref{eq:Dtnbound2}), and (\ref{eq:Dtnbound3}), 
\begin{equation} \label{eq:TOmegabound}
\Vert (T_n(w)^{-1})^{s,t} - (\Omega_{n,\delta}(w))^{s,t}\|_{1,\text{block}}^{q\times q} =
\max_{1\leq t \leq n} D_{t,n} = O(n^{-d}), \quad n \to \infty.
\end{equation}
Moreover, since $T_n(w)^{-1}$ and $\Omega_{n,\delta}(w)$ are Hermitian, the above is also true with $\| \cdot\|_{\infty, \text{block}}^{q\times q}$ norm. Thus, we get the desired results. \hfill $\Box$

\section*{Acknowledgements}
Junho Yang gratefully acknowledges the support of Taiwan's National Science and Technology Council (grant 110-2118-M-001-014-MY3).

\bibliographystyle{plainnat}
\bibliography{bib_toeplitz}
\end{document}